\documentclass{article}

\usepackage{graphicx}
\graphicspath{ {images/} }

\usepackage{amsmath}
\usepackage{amsfonts}
\usepackage{amsthm}
\usepackage{amssymb}
\usepackage{color}
\usepackage{tikz}

\newtheorem{theorem}{Theorem}[section]
\newtheorem{lemma}[theorem]{Lemma}

\newtheorem{proposition}[theorem]{Proposition}

\newtheorem{problem}[theorem]{Problem}

\newtheorem{corollary}[theorem]{Corollary}

\theoremstyle{definition}
\newtheorem{definition}[theorem]{Definition}
\theoremstyle{remark}
\newtheorem{remark}[theorem]{Remark}

\numberwithin{equation}{section}
\numberwithin{figure}{section}

\newcommand\remove[1]{}

\def\N{{\cal{N}}}
\def\M{{\cal{M}}}

\def\f2{\mathbb{F}_2}

\def\dist{\hskip0.02cm{\rm dist}\hskip0.01cm}

\newcommand{\ep}{\varepsilon}

\newcommand{\diam}{{\rm diam}\hskip0.02cm}

\newcommand{\bbN}{\mathbb{N}}
\newcommand{\al}{\alpha}
\newcommand{\be}{\beta}
\newcommand{\g}{\gamma}
\newcommand{\de}{\delta}
\newcommand{\e}{\varepsilon}

\newcommand{\la}{\lambda}
\newcommand{\s}{\sigma}

\newcommand{\bbR}{\mathbb{R}}

\newcommand{\Span}{\operatorname{span}}

\newcommand{\comp}{\operatorname{comp}}
\newcommand{\disp}{\displaystyle}
\newcommand{\lb}{\label}

\newcommand{\lra}{\longrightarrow}

\newcommand{\Buo}{Without loss of generality }

\newcommand{\DEF}{\buildrel {\mbox{\tiny def}}\over =}

\begin{document}

\title{\LARGE Metric spaces admitting
low-distortion embeddings into all $n$-dimensional Banach spaces}

\author{Mikhail I. Ostrovskii and Beata Randrianantoanina}

%\date{\today}
\maketitle

\begin{large}

\begin{abstract}
For a fixed $K\gg 1$ and
$n\in\mathbb{N}$, $n\gg 1$, we study  metric
spaces which admit embeddings with distortion $\le K$ into each
$n$-dimensional Banach space. Classical examples include spaces embeddable
into $\log n$-dimensional Euclidean spaces, and equilateral spaces.

We prove that good embeddability properties are preserved under
the operation of metric composition of metric spaces. In
particular, we prove that $n$-point ultrametrics can be
embedded with uniformly bounded distortions into arbitrary Banach
spaces of dimension $\log n$.

The main result of the paper is a new example of a family of
finite metric spaces which are not metric compositions of
classical examples and which do embed with uniformly bounded
distortion into any Banach space of dimension $n$. This partially
answers a question of G.~Schechtman.
\end{abstract}


\tableofcontents

\section{Introduction}

This paper is devoted to the following problem suggested by
Gideon~Schechtman during the Workshop in Analysis and Probability at Texas A\&M University, College Station, Texas, July 2013:

\begin{problem}\label{P:Gideon} Fix a constant $K\gg 1$ and
$n\in\mathbb{N}$ satisfying $n\gg 1$. Characterize all metric
spaces admitting embeddings with distortion $\le K$ into each
$n$-dimensional Banach space.
\end{problem}

The {\it distortion} of an (injective) embedding $f : X \to Y$ of
a metric space $(X, d_X)$ into a metric space $(Y, d_Y)$ is
defined by
\[ \dist(f) = \sup_{\substack{x,y\in X\\ x\ne y}} \frac{d_Y (f(x), f(y))}{d_X(x, y)}\cdot
 \sup_{\substack{x,y\in X\\ x\ne y}} \frac{d_X(x, y)}{d_Y (f(x), f(y))}.\]

If $\dist(f)\le K$, we say that $f$ is a {\it $K$-embedding}, and that  a metric space $(X,d_X)$ is   {\it $K$-embeddable} into a metric space $(Y,d_Y)$.

Problem \ref{P:Gideon} can be viewed as a part of modern Ramsey
Theory which seeks to characterize types of structures which can
be found inside arbitrary structures that are sufficiently large.
In the category of metric spaces in relations with their
embeddability into Banach spaces this work was initiated in
\cite{BFM86}. See \cite{BLMN05}, \cite{MN13}, and references
therein for important Ramsey-type results in the category of
metric spaces.\medskip

Since Problem \ref{P:Gideon} is vague and somewhat unrealistic,
Schechtman also suggested a more specific
Problem~\ref{P:Schechtman} (see below). Before stating it we need
to  list known examples of metric spaces embeddable into each
$n$-dimensional Banach space:

\begin{itemize}

\item[{\bf (A)}] Metric spaces admitting low-distortion embeddings
into $\log n$-dimensional Euclidean spaces. This class of examples
is obtained as a corollary of the fundamental Dvoretzky theorem
(see Theorem~\ref{T:DvoM}) which implies that for every $M>1$
there exists $s(M)>0$ such that for each $n\in\mathbb{N}$ the
space $\ell_2^k$ with $k\le s(M)\ln n$ can be linearly embedded
into any $n$-dimensional Banach space with distortion $\le M$. It
follows that any metric space which embeds with distortion $\le
K/M$ into such $\ell_2^k$ can be embedded with distortion $\le K$
into any $n$-dimensional Banach space. See Section
\ref{S:Euclidean} for more details.

\item[{\bf (B)}] A metric space is called {\it equilateral} if the
distances between all pairs of distinct points in it are equal to
the same positive number. An equilateral space is also called an
{\it equilateral set}.  Equilateral spaces of size $\le a^n$,
where $a$ depends on $K$, form another class of metric spaces
satisfying the conditions of Problem \ref{P:Gideon}. In
Section~\ref{S:Equilateral} we describe the known estimates for
the distortion of embeddings of equilateral spaces and give a
simple direct proof of the estimates that we will use for our
results.

\item[{\bf (C)}] Metric spaces from the classes mentioned in {\bf
(A)} and {\bf (B)} can be combined using a general construction,
called {\it metric composition} (see Definition~\ref{D:MetrComp}),
which was introduced in \cite{BLMN05}. In
Section~\ref{S:Composition} we present a detailed study of
embeddability properties of metric compositions of metric spaces.
We prove that for a suitable choice of parameters, the metric
composition of metric spaces which embed well into a given Banach
space $E$, also embeds well into $E$ (Theorem~\ref{T:Mbeta},
Corollary~\ref{comp}). Applying this general construction to
examples described in {\bf (A)} and {\bf (B)} we get more examples
of metric spaces satisfying the conditions of Problem
\ref{P:Gideon}. In particular ultrametrics  (see
Section~\ref{S:HST&UM}) can be obtained as metric compositions of
equilateral sets described in {\bf (B)}.  Thus, as a corollary of
our results, we obtain that ultrametrics  of exponential size
embed with uniformly bounded distortion into any $n$-dimensional
Banach space. We also provide a direct proof of this fact (see
Proposition~\ref{P:kHST} and Corollary~\ref{C:UM}).

\end{itemize}

\begin{problem}[Schechtman]\label{P:Schechtman} Can one suggest examples satisfying the condition of Problem \ref{P:Gideon} which are completely
different from the ones mentioned in {\bf (A)--(C)}?
\end{problem}

The main goal of this paper is to give an answer to this problem, that
is, to present an example of a family of metric spaces which satisfy the
condition of Problem \ref{P:Gideon}, but do not belong to any of
the classes {\bf (A)}-{\bf (C)}.%\medskip

In Section~\ref{S:SnowDiam} we show an example of a family of
graphs, which we call {\it weighted diamonds $W_n$}, which we
prove do not arise from any of the discussed above examples, and
which embed with uniformly bounded distortion into any Banach space of
dimension at least $\exp(c(\log \log |W_n|)^2)$ for a suitably chosen
$c>0$. The family $\{W_n\}$ was first constructed in \cite{Ost14},
as an example of a family of topologically complicated
series-parallel graphs which  embed with  uniformly bounded distortion
into $\ell_2$ (of infinite dimension). In the present paper we
show that $W_n$'s are snowflaked versions of standard diamonds
$D_n$'s. However $D_n$'s are not uniformly doubling so the
embeddability results for $W_n$ proved in \cite{Ost14} and in the
present paper do not follow from Assouad's theorem (see Remark
\ref{R:Ass83&Ost14}).   Our method of proof of the embeddability
of $W_n$'s uses a mixture of $\de$-net arguments and some
``linear'' manipulations.

In Section~\ref{S:coral} we show a more general construction of
`hierarchically built weighted graphs' which can have
topologically very complicated structure and which  embed with
uniformly bounded distortion into any Banach space of specified
dimension (see Theorem~\ref{T:GenTwoWay}).

\section{Low-dimensional Euclidean subsets, equilateral spaces,
and their combinations}\label{S:Equi&Eucl&Mix}

\subsection{Low-dimensional Euclidean subsets}\label{S:Euclidean}

We remind the following improvement of the Dvoretzky theorem
\cite{Dvo61}, which is due to Milman \cite{Mil71} (we state it in
a somewhat unconventional way).

\begin{theorem}\label{T:DvoM} For each $M\in(1,\infty)$ exists
$s(M)\in(0,\infty)$ such that any $n$-dimen\-sio\-nal Banach space
$X$ contains a $k$-dimensional subspace $X_k$ with $k\ge s(M)\ln
n$, such that $X_k$ is linearly isomorphic with $\ell_2^k$ with distortion $\le M$.
\end{theorem}

We assume that $s(M)$ is chosen in the optimal way. Let $M\le K$.
Theorem \ref{T:DvoM} immediately implies that any metric space
which admits an embedding into $\ell_2^k$, where $k=\lceil
s\left({M}\right)\ln n\rceil$, with distortion $\le\frac{K}{M}$,
also admits an embedding with distortion $\le K$ into any
$n$-dimensional Banach space.
\medskip

This relates our study with the following major open problem of
the theory of metric embeddings: {\it Find an intrinsic
characterization of those separable metric spaces $(X, d_X)$ that
admit bilipschitz embeddings into $\ell_2^n$ for some $n\in\bbN$.}
See \cite{Sem99,LP01,Hei03} for a discussion of this problem. One
of the most important results on this problem is the Assouad
theorem, which we mention below (Theorem \ref{T:Assouad}).
However, it should be mentioned that in the present context we are
interested in the version of the problem for which the dimension
is specified, see very interesting recent results related to this
problem in \cite{NN12,DS13}, and related comments in \cite[Remark
3.16]{Hei03}.

\subsection{Equilateral spaces}\label{S:Equilateral}

It follows from standard volumetric estimates that the maximal
$\de$-separated set, for $\de>0$, contained in the unit ball of
any $n$-dimensional Banach space has cardinality at least
$\de^{-n}$.  It is also easy to see that a bijection between
any $\de$-separated set in the unit ball and an equilateral space
of the same cardinality has distortion $\le(2/\de)$. Therefore an
equilateral space of size $\le\left(\frac{K}2\right)^n$ admits an
embedding with distortion $\le K$ into an arbitrary
$n$-dimensional Banach space.

A standard  volumetric estimate also  gives an upper bound on the
cardinality of an equilateral space which can be embedded with
distortion at most $K$ in an $n$-dimensional Euclidean space.

\begin{lemma}\label{L:lowerequi} Let $\{a_i\}_{i=1}^{N}$
be an equilateral space which is $K$-embeddable into an
$n$-dimensional Banach space. Then $N\le (2K+1)^n$.
\end{lemma}

\begin{proof} Let $\{a_i\}_{i=1}^{N}$ be equilateral with distances equal to $1$, and $\varphi: \{a_i\}_{i=1}^{N}\to X$ be a $K$-embedding.
We may assume that there is $h>0$ such that
\[h\le ||\varphi(a_i)-\varphi(a_j)||\le Kh.\]
Therefore $N\left(\frac h2\right)^n<\left(Kh+\frac h2\right)^n$
and $N \le (2K+1)^n$.\end{proof}

The most precise, known, estimates for the distortion of an
embedding of   exponential size equilateral spaces in any Banach
space can be obtained using the following result of
Arias-de-Reyna, Ball, and Villa \cite{ABV98} (an earlier version
of this result was proved by Bourgain, see \cite[Theorem
4.3]{FL94}, using Milman's \cite{Mil85} quotient of subspace
theorem).

\begin{theorem}[{\cite{ABV98}}]\label{T:ABV}
Let $\e>0$, $X$ be an $n$-dimensional Banach space, $B$ its closed
unit ball and $\mu$, the Lebesgue measure on $B$ normalized so
that $\mu(B)=1$. If $t=\sqrt{2}\,(1-\ep)$, then
\begin{equation*}\label{E:ABV}
\mu\otimes\mu\{(x,y)\in B\times B:~||x-y||\le t\}\le
(1-\ep^2(2-\ep)^2)^{\frac{n}2}.
\end{equation*}
\end{theorem}

Theorem \ref{T:ABV} implies that the set $L$ of points $x$ in $B$
for which
\[\mu\{y\in B:~||x-y||\le t\}\le
2(1-\ep^2(2-\ep)^2)^{\frac{n}2},\] satisfies $\mu(L)\ge\frac12$.
Choosing $t$-separated points in the set $L$, one by one, we get a
$t$-separated set of cardinality at least
\[\frac1{4\,(1-\ep^2(2-\ep)^2)^{\frac{n}2}}.\]
This implies
\begin{corollary}\label{bestdist}
For every $s>\sqrt{2}$ there exists $C(s)~
(=\ln(s^2/2\sqrt{s^2-1}))$ so that
 an equilateral set of size $\frac14\exp(C(s)n)$ embeds in any
$n$-dimensional Banach space $X$ with distortion $\le
s$.
\end{corollary}

It
is a long standing open problem whether $\sqrt{2}$ in Corollary~\ref{bestdist} can be replaced by 1. That is

\begin{problem}\label{P:AlmEQ}
Does there exist a function $C:(0,1)\to (0,\infty)$  such that for
each $\varepsilon\in(0,1)$ and each $n\in\bbN$, an equilateral
space of size $\exp(C(\e)n)$ embeds in any $n$-dimensional Banach
space $X$ with distortion $\le 1+\varepsilon$?
\end{problem}

The answer is known to be positive for Banach space $X$ with a
1-subsymmetric basis \cite{BBK89}, for uniformly convex Banach
spaces \cite{ABV98} (the function $C(\ep)$ depends on the modulus
of convexity), and in some other cases, see \cite{BB91,BPS95}.

In the sequel we will use bilipschitz embeddings of equilateral
sets into unit spheres of finite-dimensional Banach spaces. For
completeness we include a simple proof with the specific constants
that we will use.

\begin{lemma}\label{L:apart}
There exists a constant $\delta\ge\frac1{16}$, so that for every
$m\ge 1$ the unit sphere of every $m$-dimensional Banach space $X$
contains $2\cdot4^{m-1}$ elements of  mutual distance at least $\delta$.
\end{lemma}

\begin{proof} Let $\delta>0$ and $T$ be the maximal $\delta$-separated set on
$S_X$. Then $(2\delta)$-balls centered at $T$ cover the set of
points between the spheres of radius $(1+\delta)$ and
$(1-\delta)$. Therefore the cardinality $|T|$ of $T$ satisfies:

\[|T|\ge\frac{(1+\delta)^m-(1-\delta)^m}{(2\delta)^m}.\]

So we need $\delta$ such that the following inequality holds for
all $m\ge 1$:
\[\frac12(8\delta)^m\le (1+\delta)^m-(1-\delta)^m.\]

Let $\delta=\frac1{16}$. Since the right-hand side is
$\ge2m\delta$, the conclusion follows for $m\ge 2$. For $m=1$ the
conclusion is obvious.
\end{proof}

\begin{corollary}\label{C:ABitMore} If we replace  $2\cdot 4^{m-1}$
by $2^m$ in the statement of Lemma~\ref{L:apart},  we can  take
${\disp \de=1/8}$.
\end{corollary}

\subsection{Metric compositions}\label{S:Composition}

The following general construction of combining metric spaces was
introduced by Bartal, Linial, Mendel and Naor for their study of
metric Ramsey-type phenomena \cite{BLMN05}.

\begin{definition}[Metric composition, \cite{BLMN05}]\label{D:MetrComp}
Let $M$ be a finite metric space. Suppose that there is a
collection of disjoint finite metric spaces $N_x$ associated with
the elements $x$ of $M$. Let $\N = \{ N_x \}_{x \in M}$. For
$\beta\geq 1/2$, the {\it $\beta$-composition} of $M$ and $\N$,
denoted by $M_\beta[\N]$, is a metric space on the disjoint union
$\dot {\boldsymbol{\cup}}_x N_x$. Distances in $M_\beta[\N]$ are
denoted $d_\be$ and defined as follows. Let $x,y \in M$ and $u \in
N_x, v \in N_y$; then:
$$ d_\be(u,v)= \begin{cases} d_{N_x}(u,v) & x=y \\
   \beta \gamma d_M(x,y) & x\neq y,\end{cases}  $$ where $\gamma=\frac{\max_{z \in M}
\diam(N_z)}{\min_{x\neq y \in M} d_M(x,y)}$. It is easily checked
that the choice of the factor $\beta\gamma$ guarantees that
$d_\be$ is indeed a metric.
\end{definition}

We prove that the metric composition preserves embeddability
properties of metric spaces in the following sense.

\begin{theorem}\label{T:Mbeta}
Let $C,D\ge1$. Let $E$ be a Banach space and $M$ be a finite
metric space which is $C$-embeddable in $E$. Let
${\N}=\{N_x\}_{x\in M}$ be a family of finite metric spaces which
are $D$-embeddable in $E$. Let $\be>2(C+1)$. Then $M_\be[\N]$ is
$A$-embeddable in $E$, where ${\disp A=\max\left(D,
C+\frac{2C(C+1)}{\be-2(C+1)}\right)}$.

In particular, if $\e>0$, $D\le C$, and ${\disp
\be>2\left(\frac{C+\e}{\e}\right)(C+1)}$ then $M_\be[\N]$ is
$(C+\e)$-embeddable in $E$.
\end{theorem}

\begin{proof}
\Buo we can and do assume that
\begin{equation}\label{E:AtLeast1}\min_{x\ne y\in M}d_M(x,y)=1.\end{equation}
Let $\g=\max_{x\in M}(\diam N_x).$ We denote the metric in
$M_\be[\N]$ by $d_\be$. For all $y_1,y_2\in M_\be[\N]$ there exist
$x_1,x_2\in M$ so that $y_1\in N_{x_1}$, $y_2\in N_{x_2}$. We have
\[d_\be(y_1,y_2)=\begin{cases} d_{N_{x}}(y_1,y_2) &\text{\rm if}\ \ x_1=x_2=x,\\
\be\g d_M(x_1,x_2) &\text{\rm if}\ \ x_1\ne x_2.
\end{cases}
\]

By assumption, there exists $\Psi:M\lra E$ so that for all
$x_1,x_2\in M$,
\[ \frac1C d_M(x_1,x_2) \le \|\Psi(x_1)-\Psi(x_2)\| \le d_M(x_1,x_2). \]

Also, for all $x\in M$ there exist $\Phi_x:N_x\lra E$,  so that
for all $y, y_1,y_2\in N_x$
\begin{equation}\label{lipphix}
 \frac1D d_{N_{x}}(y_1,y_2) \le \|\Phi_x(y_1)-\Phi_x(y_2)\| \le d_{N_{x}}(y_1,y_2),
\end{equation}
and
\begin{equation}\label{phidiam}
\|\Phi_x(y)\| \le \diam(N_x).
\end{equation}

Let $\la=\be -2$, and define $\Phi:M_\be[\N]\lra E$ by
\[\Phi(y)=\Phi_x(y)+\la\g\Psi(x),\]
if $y\in N_x$ for some $x\in M$.

We claim that $\Phi$ is an $A$-embedding, where  ${\disp
A=\max\left(D, C+\frac{2C(C+1)}{\be-2(C+1)}\right)}$.

First we consider  $y_1,y_2\in M_\be[\N]$ so that there
exists $x \in M$ with $y_1,y_2\in N_{x}$. Then
$d_\be(y_1,y_2)=d_{N_{x}}(y_1,y_2)$, and
\[
\Phi(y_1)-\Phi(y_2)=\Phi_x(y_1)+\la\g\Psi(x)-(\Phi_x(y_2)+\la\g\Psi(x))
=\Phi_x(y_1)-\Phi_x(y_2).\] Thus by \eqref{lipphix} we have
\begin{equation}\label{lipphisamex}
 \frac1D d_\be(y_1,y_2) \le \|\Phi(y_1)-\Phi(y_2)\| \le d_{\be}(y_1,y_2).
\end{equation}

Next we consider  $y_1,y_2\in M_\be[\N]$ so that $y_1\in
N_{x_1}$, $y_2\in N_{x_2}$, for some $x_1,x_2\in M$, with $x_1\ne
x_2$.  Then $d_\be(y_1,y_2)=\be\g d_M(x_1,x_2)$, and
\[
\begin{split}
\|\Phi(y_1)-\Phi(y_2)\|&=\|\Phi_{x_1}(y_1)+\la\g\Psi(x_1)-(\Phi_{x_2}(y_2)+\la\g\Psi(x_2))\|\\
&\le \|\Phi_{x_1}(y_1)\|+\|\Phi_{x_2}(y_2)\|+\la\g\|\Psi(x_1)-\Psi(x_2)\|\\
&\le 2\g + \la\g d_M(x_1,x_2) \stackrel{\mbox{\tiny  {\rm by }\eqref{E:AtLeast1}}}{\le}
(2\g + \la\g) d_M(x_1,x_2)\\
&=\be\g d_M(x_1,x_2)=d_\be(y_1,y_2).
\end{split}
\]

On the other hand
\[
\begin{split}
\|\Phi(y_1)-\Phi(y_2)\|&=\|\Phi_{x_1}(y_1)+\la\g\Psi(x_1)-(\Phi_{x_2}(y_2)+\la\g\Psi(x_2))\|\\
&\ge\la\g\|\Psi(x_1)-\Psi(x_2)\|- \|\Phi_{x_1}(y_1)\|-\|\Phi_{x_2}(y_2)\|\\
&\ge  \frac1C\la\g d_M(x_1,x_2)-2\g\stackrel{\mbox{\tiny  {\rm by }\eqref{E:AtLeast1}}}{\ge} \left(\frac1C\la - 2\right)\g d_M(x_1,x_2)\\
&=\left(\frac1C(\be-2) - 2\right)\g d_M(x_1,x_2)
\\
&=\frac{\be-2C - 2}{C\be} \be\g d_M(x_1,x_2)\\
&=\frac{\be-2C - 2}{C\be} d_\be(y_1,y_2).
\end{split}
\]

This, together with \eqref{lipphisamex}, implies that $\Phi$ is an
$A$-embedding of $ M_\be[\N]$ into $E$.

Note that if ${\disp \be>2\left(\frac{C+\e}{\e}\right)(C+1)}$ then
\[
\begin{split}
\frac{C\be}{\be-2C - 2}&=C+\frac{2C(C+1)}{\be-2C - 2}< C+\frac{2C(C+1)}{2(1+\frac C\e)(C+1)-2(C+1)}=C+\e.
\end{split}
\]

Thus, if $D\le C$, then $A\le C+\e$.
\end{proof}

\begin{definition}\cite{BLMN05}\label{D:Comp}
Given a class $\cal{M}$ of finite metric spaces, and $\be\ge1$, we
define its closure $\comp_\be(\cal{M})$ under $\ge
\be$-compositions as the smallest class $\cal{C}$ of metric spaces
that contains all spaces in $\cal{M}$ and satisfies the following
condition: Let $M\in\cal{M}$, $\be'\ge\be$, and associate with
every $x\in M$ a metric space $N_x$ that is isometric to a space
in $\cal{C}$. Then the space $M_{\be'}[\N]$ is also in
$\cal{C}$.
\end{definition}

Note that $\comp_\be(\cal{M})$ can be described as a union of
smaller classes which have increasing complexity. More precisely
\begin{equation}\lb{complexity}
\comp_\be({\cal{M}})=\bigcup_{m=0}^\infty {\cal{C}}_m,
\end{equation}
where ${\cal{C}}_0=\cal{M}$, and for $m\in\bbN$, ${\cal{C}}_m$ is
the class of metric spaces of the form $M_{\be'}[\N]$, where
$M\in\cal{M}$, $\be'\ge\be$, and ${\N}=\{N_x\}_{x\in M}$,
where  for every $x\in M$ a metric space $N_x$  is isometric to a
space in $\bigcup_{i=0}^{m-1}{\cal{C}}_{i}$.

The operation of composition has the following associativity
property.

\begin{proposition}\label{P:trans-comp}
Let $M$ be a finite metric space, $\be_1\ge 1/2, \be_2\ge 1$,
$\N_1=\{N_x\}_{x\in M}$ and $\N_2=\{\tilde{N}_y\}_{y\in
M_{\be_1}(\N_1)}$ be   families of finite metric spaces. Then
\[\left(M_{\be_1}(\N_1)\right)_{\be_2}(\N_2)=M_{\be_1}(\N_3),\]
where $\N_3$ is a  family  of finite metric spaces of the form  $(N_x)_{\be_x}(\{\tilde{N}_{y}\}_{y\in N_x})$, where $x\in M$ and $\be_x\ge\be_2$.
\end{proposition}

\begin{proof}
For each $x\in M$, define $N^\star_x$  as the disjoint union of $\tilde{N}_y$ over $y\in N_x$, with the metric inherited from $\left(M_{\be_1}(\N_1)\right)_{\be_2}(\N_2)$. Let $\N_3$ be the collection of all $N^\star_x$
\[\N_3=\{N^\star_x\}_{x\in M}=\{\dot
{\boldsymbol{\cup}}_{y\in N_x} \tilde{N}_y\}_{x\in M}.\]

Denote the metric in $\left(M_{\be_1}(\N_1)\right)_{\be_2}(\N_2)$
by $d$. For all $z_1, z_2\in
\left(M_{\be_1}(\N_1)\right)_{\be_2}(\N_2)$ there exist $y_1,
y_2\in M_{\be_1}(\N_1)$ so that $z_1\in \tilde{N}_{y_1}$, $z_2\in
\tilde{N}_{y_2}$, and $x_1, x_2 \in M$ so that $y_1\in N_{x_1}$,
$y_2\in N_{x_2}$. We have
\[d(z_1,z_2)=\begin{cases} d_{\tilde{N}_{y}}(z_1,z_2) \ \ &\text{\rm if } z_1,z_2\in \tilde{N}_{y};\\
\be_2\g_2 d_{N_x}(y_1,y_2) &\text{\rm if } z_1,z_2\in N^\star_x, y_1\ne y_2;\\
\be_2\g_2 \be_1\g_1d_{M}(x_1,x_2) &\text{\rm if } z_1\in N^\star_{x_1} ,z_2\in N^\star_{x_2}, x_1\ne x_2,
\end{cases}
\]
where
\[\g_1=\frac{\max_{x\in M} \diam(N_x)}{\min_{x_1\ne x_2\in M} d_M(x_1,x_2)} ; \ \ \ \  \ \ \ \
\g_2=\frac{\max_{y\in M_{\be_1}(\N_1)} \diam(\tilde{N}_{y})}{\min_{y_1\ne y_2\in M_{\be_1}(\N_1)} d_{M_{\be_1}(\N_1)} (y_1,y_2)}.\]

We claim that for every $x\in M$, there exists $\be_x\ge \be_2$ so that
\begin{equation}\label{N3}
N^\star_x=(N_x)_{\be_x}(\{\tilde{N}_{y}\}_{y\in N_x}).
\end{equation}
Indeed, it is enough to observe that for every $x\in M$, $\g_x\le\g_2$, where
\[\g_x\DEF\frac{\max_{y\in N_x} \diam(\tilde{N}_{y})}{\min_{y_1\ne y_2\in N_x} d_{N_x} (y_1,y_2)}.\]

To prove that
\[\left(M_{\be_1}(\N_1)\right)_{\be_2}(\N_2)=M_{\be_1}(\N_3),\]
we define
\[\g_3\DEF\frac{\max_{x\in M} \diam(N^\star_x)}{\min_{x_1\ne x_2\in M} d_M(x_1,x_2)}.\]
Observe that
\[\max_{x\in M} \diam(N^\star_x)=\max\left(\be_2\g_2\max_{x\in M} \diam(N_x), \max_{y\in M_{\be_1}(\N_1)} \diam(\tilde{N}_{y})\right).\]
If
\[\max_{y\in M_{\be_1}(\N_1)} \diam(\tilde{N}_{y})>\be_2\g_2\max_{x\in M} \diam(N_x),
\]
then, by definition of $\g_2$,
\[\min_{y_1\ne y_2\in M_{\be_1}(\N_1)} d_{M_{\be_1}(\N_1)} (y_1,y_2)>\be_2\max_{x\in M} \diam(N_x),\]
and thus, since $\be_2\ge1$,
\[\min_{x\in M}\min_{y_1\ne y_2\in N_x} d_{N_x} (y_1,y_2)>\max_{x\in M} \diam(N_x),\]
which is impossible. Thus
\[\max_{x\in M} \diam(N^\star_x)=\be_2\g_2\max_{x\in M} \diam(N_x).\]
Hence
\[\frac{\be_1\be_2\g_1\g_2}{\g_3}=\be_1,\]
and the proposition is proven.
\end{proof}

As an immediate corollary of Proposition~\ref{P:trans-comp} and
\eqref{N3} we obtain.

\begin{corollary} \label{compcomp}
Let $\M$ be a family of finite metric spaces and $\be\ge 1$. Then $\comp_\be(\comp_\be(\M))= \comp_\be(\M)$.
\end{corollary}

As a consequence of Theorem~\ref{T:Mbeta} we obtain the following result.

\begin{corollary}\label{comp}
Let $C\ge1$. Let $E$ be a Banach space and $\cal{M}$ be  a family
of finite metric spaces which are $C$-embeddable in $E$. Let
$\be>2(C+1)$. Then every space in $\comp_\be(\cal{M})$ is
$A$-embeddable in $E$, where ${\disp
A=C+\frac{2C(C+1)}{\be-2(C+1)}}$.

In particular, for any $\e>0$, if ${\disp
\be>2\left(\frac{C+\e}{\e}\right)(C+1)}$ then every space in $\comp_\be(\cal{M})$ is
$(C+\e)$-embeddable in $E$.
\end{corollary}

\begin{proof}  We prove this by induction on the level of complexity of spaces in $\comp_\be(\cal{M})$.

If $M\in {\cal{C}}_0=\cal{M}$, then by assumption, $M$ is
$C$-embeddable in $E$.

Suppose that for some $m\in\bbN$, every space in
$\bigcup_{i=0}^{m-1}{\cal{C}}_{i}$ is $A$-embeddable in $E$.

Let $X\in {\cal{C}}_m$. Then there exist $M\in\cal{M}$,
$\be'\ge\be$, and ${\N}=\{N_x\}_{x\in M}$, where  for every
$x\in M$ a metric space $N_x$  is isometric to a space in
$\bigcup_{i=0}^{m-1}{\cal{C}}_{i}$, so that $X=M_{\be'}[\N].$
By assumption, $M$ is $C$-embeddable in $E$, and by inductive
hypothesis, every space in ${\N}$ is  $A$-embeddable in $E$.
Thus by Theorem~\ref{T:Mbeta}, $X$ is $B$-embeddable in $E$, where
\[
B=\max\left(A, C+\frac{2C(C+1)}{\be-2(C+1)}\right)=A.
\]
\end{proof}

\subsection{Ultrametrics and hierarchically well-separated trees}\label{S:HST&UM}

An {\em ultrametric} is a metric space $(M,d)$ such
that for every $x,y,z\in M$,
$$
d(x,z)\le \max\{d(x,y),d(y,z)\}.
$$

These spaces appeared in a natural way in the study of $p$-adic
number fields, see \cite{Sch84}. Currently ultrametrics play an
important role in many branches of mathematics, see for example
\cite{BLMN05}, \cite{Hug12}, \cite{MN13}, and references therein.
It is known that ultrametrics have very good embedding properties,
see \cite{Shk04} and its references. In particular, Shkarin
\cite{Shk04} proved that for any finite ultrametric $(M, d)$,
there exists $m = m(M, d) \in\bbN$ such that for any Banach space
$E$ with $\dim E\ge m$ there exists an isometric embedding of $M$
into $E$. In this result $m$ is large and depends on $M$, not only
on the cardinality of $M$. Observe that any isometric embedding of
an equilateral space with $n$ points (it is a simplest
ultrametric) into a Euclidean space requires dimension $\ge n-1$.
So isometric embeddings of ultrametrics require large dimension.
The situation changes if we allow some distortion: Bartal, Linial,
Mendel and Naor \cite{BLMN04} proved that  there exist constants
$C\ge 1$ and $c>0$ such that any $n$-point ultrametric $C$-embeds
  into $\ell_p^{k}$, for any $k\ge
c\ln n$ and any $1\le p\le \infty$. The goal of this section
is to prove that there exists a universal constant $K$ such that
any $n$-point ultrametric embeds into any Banach space of
dimension $\log_2 n$ with distortion $\le K$.

It turns out that  for embeddability of ultrametrics it is
convenient to use the following, more restricted class of metrics.

\begin{definition}[\cite{Bar96}]\label{D:HST}
For $k\geq 1$, a $k$-\emph{hierarchically well-separated tree}
($k$-HST) is a metric space whose elements are the leaves of a
rooted tree $T$. To each vertex $u\in T$ there is associated a
label $\Delta(u) \ge 0$ such that $\Delta(u)=0$ if and only if $u$
is a leaf of $T$. It is required that if a vertex $u$ is a child
of a vertex $v$ then $\Delta(u)\leq \Delta(v)/k$. The distance
between two leaves $x,y\in T$ is defined as
$\Delta(\hbox{lca}(x,y))$, where $\hbox{lca}(x,y)$ is the least
common ancestor of $x$ and $y$ in $T$. A $k$-HST is said to be
\emph{exact} if $\Delta(u)=\Delta(v)/k$ for every two internal
vertices (that is, neither $u$ nor $v $ is a leaf) where $u$ is a
child of $v$.
\end{definition}

First, note that an ultrametric on a finite set and a (finite)
$1$-HST are identical concepts. Any $k$-HST is also a $1$-HST,
i.e., an ultrametric. When we discuss $k$-HST's, we freely use the
tree $T$ as in Definition~\ref{D:HST}, which we refer to as \emph{the tree defining
the} HST. An internal vertex in $T$ with out-degree $1$ is said to
be \emph{degenerate}. If $u$ is nondegenerate, then $\Delta(u)$ is
the diameter of the sub-space induced on the subtree rooted by
$u$. Degenerate nodes do not influence the metric on $T$'s leaves;
hence we may assume that all internal nodes are nondegenerate
(note that this assumption need not hold for \emph{exact}
$k$-HST's).

By \cite[Proposition~3.3]{BLMN05}, the class of $k$-HST
coincides with $\comp_k(EQ)$, where $EQ$ denotes the class of
finite equilateral spaces. Thus,  by
Corollary~\ref{C:ABitMore} and Corollary~\ref{comp},
 that there exists $k_0\ge 2$ so that every $k$-HST, with $k\ge k_0$, admits a bilipschitz embedding into any
Banach space $X$ with $\dim X\ge\log_2D$, where $D$ is the maximal
out-degree of a vertex in the tree defining the $k$-HST, with a
uniformly bounded distortion, which generalizes
\cite[Proposition~3]{BLMN04}.
\medskip

Our next goal is to provide an alternative more direct proof of
this result.

\begin{proposition}\label{P:kHST}
Any $k$-HST with $k> 17$ admits a bilipschitz embedding into any
Banach space $X$ with $\dim X\ge\log_2D$, where $D$ is the maximal
out-degree of a vertex in the tree defining the $k$-HST, with
distortion not exceeding $\displaystyle{\frac{16k}{k-17}}$.
\end{proposition}

\begin{proof} Let $v$ be any of the non-leaf vertices of the tree defining the
$k$-HST. The number of edges $E(v)$ contributing to the out-degree
of $v$ is at most $D$. Let $\de=1/8$. By
Corollary~\ref{C:ABitMore} there is an injective map $\varphi_v$
from $E(v)$ to the $\delta$-net on the unit sphere of $X$.
Combining $\varphi_v$ for all non-leaf vertices $v$ we get a (no
longer injective) map $\varphi$ from the edge set of the tree to
the $\delta$-net.

Now we define the map $f$ on the set of leaves into $X$, for a
leaf $\ell$ let
\[f(\ell)=\sum_{t\in[r,\bar\ell]}\Delta(t)\varphi(t\tilde t),
\]
where $r$ is the root, $\bar\ell$ is the last non-leaf on the way
from $r$ to $\ell$, $[r,\bar\ell]$ is the path joining $r$ and
$\bar\ell$ (path is regarded as a set of vertices), $\tilde t$ is
the next after $t$ vertex on the path from $r$ to $\ell$, and
$\Delta(t)$ is the label assigned according to
Definition~\ref{D:HST}.

Let us estimate the distortion. Let $\ell_1$ and $\ell_2$ be two
leaves in the tree, $v$ be their least common ancestor, $v_1$ and
$v_2$ be the first (after $v$) vertices on the paths $[v,\ell_1]$
and $[v,\ell_2]$, respectively. Then (we use the fact that
$\Delta(v)=d(\ell_1,\ell_2)$)
\[\begin{split}||f(\ell_1)&-f(\ell_2)||=\left\|\sum_{t\in[v,\bar\ell_1]}\Delta(t)\varphi(t\tilde
t)-\sum_{t\in[v,\bar\ell_2]}\Delta(t)\varphi(t\tilde
t)\right\|\\
&\ge\Delta(v)||\varphi(vv_1)-\varphi(vv_2)||-
\sum_{t\in[v_1,\bar\ell_1]}\Delta(t)||\varphi(t\tilde
t)||-\sum_{t\in[v_2,\bar\ell_2]}\Delta(t)||\varphi(t\tilde t)||\\
&\ge\frac1{8}\Delta(v)-2\Delta(v)\left(\frac1k+\frac1{k^2}+\frac1{k^3}+\dots\right)\\
&=d(\ell_1,\ell_2)\left(\frac1{8}-\frac2{k-1}\right)
\end{split}
\]
On the other hand,
\[\begin{split}||f(\ell_1)&-f(\ell_2)||=\left\|\sum_{t\in[v,\bar\ell_1]}\Delta(t)\varphi(t\tilde
t)-\sum_{t\in[v,\bar\ell_2]}\Delta(t)\varphi(t\tilde
t)\right\|\\
&\le\sum_{t\in[v,\bar\ell_1]}\Delta(t)||\varphi(t\tilde
t)||+\sum_{t\in[v,\bar\ell_2]}\Delta(t)||\varphi(t\tilde t)||\\
&\le2\Delta(v)\left(1+\frac1k+\frac1{k^2}+\frac1{k^3}+\dots\right)\\
&=d(\ell_1,\ell_2)\frac{2k}{k-1}.\qedhere
\end{split}
\]
\end{proof}

Since for any $k>1$, any ultrametric is $k$-bilipschitz equivalent
to a $k$-HST (\cite{Bar99}, see also \cite[Lemma 3.5]{BLMN05}), we
obtain the following corollary of Proposition \ref{P:kHST} (which
is an arbitrary-space version of results of \cite{BLMN04} and
\cite{BM04}).

\begin{corollary}\label{C:UM}
Any $n$-point ultrametric  embeds with uniformly bounded
distortion into any Banach space $X$ with $\dim(X)\ge \log_2n$.
\end{corollary}

\begin{remark} It is natural to try to achieve distortions
arbitrarily close to $1$ in Proposition~\ref{P:kHST}, provided
that $k$ is sufficiently large and the dimension is a sufficiently
large multiple of $\log_2D$. This is what was done for embeddings
into $\ell_p$ in \cite{BLMN04}, as a consequence of the fact that
$n$-point equilateral sets can be $(1+\e)$-embedded into
$\ell_p^{k}$ with $k\le C(\ep)\ln n$.  By
Corollary~\ref{comp} (or a careful reading of the proof of
Proposition~\ref{P:kHST}) we obtain the same conclusion in every
Banach space that satisfies the condition of
Problem~\ref{P:AlmEQ}, see Section~\ref{S:Equilateral} for a list
of classes of spaces known to satisfy this condition.
\end{remark}

\section{A new example: weighted diamond graphs}\label{S:SnowDiam}

Our basic example is the family of weighted diamonds
$\{W_n\}_{n=0}^\infty$ introduced in \cite{Ost14}. Let us recall
the definitions.

\begin{definition}[\cite{GNRS04}]\label{D:Diamonds}
Diamond graphs $\{D_n\}_{n=0}^\infty$ are defined as follows: The
{\it diamond graph} of level $0$ is denoted $D_0$, it contains two
vertices joined by an edge. The {\it diamond graph} $D_n$ is
obtained from $D_{n-1}$ as follows. Given an edge $uv\in
E(D_{n-1})$, it is replaced by a quadrilateral $u, a, v, b$, with
edges $ua$, $av$, $vb$, $bu$. (See Figure~\ref{F:Diamond2} for a
sketch of $D_2$.)
\end{definition}

%\remove{

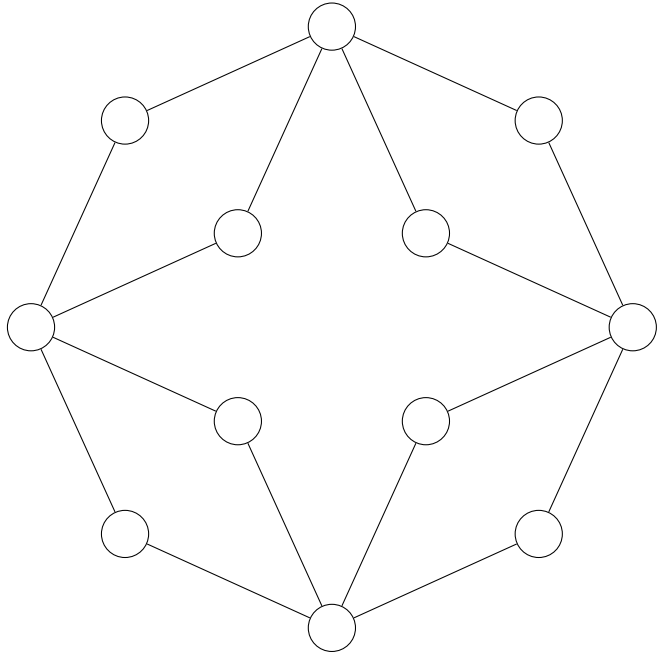
\begin{figure}
\begin{center}
{
\begin{tikzpicture}
  [scale=.25,auto=left,every node/.style={circle,draw}]
  \node (n1) at (16,0) {\hbox{~~~}};
  \node (n2) at (5,5)  {\hbox{~~~}};
  \node (n3) at (11,11)  {\hbox{~~~}};
  \node (n4) at (0,16) {\hbox{~~~}};
  \node (n5) at (5,27)  {\hbox{~~~}};
  \node (n6) at (11,21)  {\hbox{~~~}};
  \node (n7) at (16,32) {\hbox{~~~}};
  \node (n8) at (21,21)  {\hbox{~~~}};
  \node (n9) at (27,27)  {\hbox{~~~}};
  \node (n10) at (32,16) {\hbox{~~~}};
  \node (n11) at (21,11)  {\hbox{~~~}};
  \node (n12) at (27,5)  {\hbox{~~~}};

  \foreach \from/\to in {n1/n2,n1/n3,n2/n4,n3/n4,n4/n5,n4/n6,n6/n7,n5/n7,n7/n8,n7/n9,n8/n10,n9/n10,n10/n11,n10/n12,n11/n1,n12/n1}
    \draw (\from) -- (\to);

\end{tikzpicture}
} \caption{Diamond $D_2$.}\label{F:Diamond2}
\end{center}
\end{figure}
%}

\begin{definition}[\cite{Ost14}]\label{D:W_k}
We pick a number $\ep\in\left(0,\frac12\right)$. The sequence
$\{W_n\}_{n=0}^\infty$ of {\it weighted diamonds} is defined in
terms of diamonds $\{D_n\}_{n=0}^\infty$ as follows (see Figure
\ref{F:GraphW} for a sketch of $W_2$):

\begin{itemize}

\item $W_0$ is the same as $D_0$. The only edge of $D_0$ is given
weight $1$.

\item $W_1=D_1\cup W_0$ with edges of $D_1$ given weights
$\left(\frac12+\ep\right)$; weight of the edge of $W_0$ stays as
$1$ (as it was in the first step of the construction).

\item $W_2=D_2\cup W_1$ with edges  of $D_2$ given weights
$\left(\frac12+\ep\right)^2$; weights of the edges of $W_{1}$ stay
as they were in the previous step of the construction.

\item \dots.

\item $W_n=D_n\cup W_{n-1}$ with edges of $D_n$ given weights
$\left(\frac12+\ep\right)^n$; weights of the edges of  $W_{n-1}$
stay as they were in the previous step of the construction.

\item Graphs $\{W_n\}$ are endowed with their shortest path
distances which we denote $d_W$.

\end{itemize}

\end{definition}

\begin{remark} Observe that the metric of $W_n$ depends on $\ep$
although we do not reflect this fact in our notation.
\end{remark}

Note that $D_k$ has $4^k$ edges and that in each step we introduce
$2$ new vertices of $W_k$ per each edge of $D_{k-1}$. Hence $W_k$,
$k\ge 1$, has $2(4^{k-1}+4^{k-2}+\dots+1)+2$ vertices. Thus
${\disp \frac 124^n\le |W_n|< 4^n}$ for $n\ge 2$, and
\begin{equation}\label{cardwn} 2n-1\le \log_2|W_n|<2n.
\end{equation}

%\remove{

\begin{figure}
\begin{center}
{
\begin{tikzpicture}
  [scale=.25,auto=left,every node/.style={circle,draw}]
  \node (n1) at (16,0) {\hbox{~~~}};
  \node (n2) at (5,5)  {\hbox{~~~}};
  \node (n3) at (11,11)  {\hbox{~~~}};
  \node (n4) at (0,16) {\hbox{~~~}};
  \node (n5) at (5,27)  {\hbox{~~~}};
  \node (n6) at (11,21)  {\hbox{~~~}};
  \node (n7) at (16,32) {\hbox{~~~}};
  \node (n8) at (21,21)  {\hbox{~~~}};
  \node (n9) at (27,27)  {\hbox{~~~}};
  \node (n10) at (32,16) {\hbox{~~~}};
  \node (n11) at (21,11)  {\hbox{~~~}};
  \node (n12) at (27,5)  {\hbox{~~~}};

  \foreach \from/\to in
  {n1/n2,n1/n3,n2/n4,n3/n4,n4/n5,n4/n6,n6/n7,n5/n7,n7/n8,n7/n9,n8/n10,n9/n10,n10/n11,n10/n12,n11/n1,n12/n1,n1/n7,n1/n4,n4/n7,n7/n10,n10/n1}
    \draw (\from) -- (\to);

\end{tikzpicture}
}
\end{center}
\caption{Graph $W_2$. The longest edge has weight $1$, the
shortest edges have weights $\left(\frac12+\ep\right)^2$, the
edges of intermediate length have weights
$\left(\frac12+\ep\right)$.}\label{F:GraphW}
\end{figure}
%}

Using a mixture of $\de$-net arguments (Lemma \ref{L:apart}) and
some ``linear'' manipulations we prove that $W_n$'s admit
bounded-distortion embeddings into all Banach spaces with
dimension bounds which are substantially smaller than the ones
implied by the Dvoretzky-type Theorem \ref{T:DvoM}. Namely,
in Section~\ref{S:EmbW_kAnyBS} we prove the following result
(Corollary~\ref{cor:anyemb-old}).

\begin{corollary}\label{cor:anyemb}
For every $\e\in (1/2,1)$, there exist  constants  $c, C>1$ so that for every $n\ge C$, $W_n$ can be embedded in every Banach
space $X$ with $\dim X\ge\exp(c(\log \log |W_n|)^2)$ with the distortion
bounded from above by a constant which depends only on
$\ep$.
\end{corollary}

Before proving Corollary~\ref{cor:anyemb} we study the structure of the weighted diamonds $W_n$. We show that they are snowflaked versions of standard diamonds $D_n$  and that  $W_n$'s are not included in the set of examples
presented in Section~\ref{S:Equi&Eucl&Mix}.

\subsection{Weighted diamonds are bilipschitz-equivalent to
snowflaked dia\-monds}

The following definition is standard (see \cite[p.~98]{Hei01}):

\begin{definition}\label{D:Snowflake}
Let $(X,d_X)$ be a metric space and $0<\alpha<1$. The space $X$
endowed with a modified metric $(d_X(u,v))^\alpha$ is called a
{\it snowflake} of $(X,d_X)$. We also say that $(X,d_X^\alpha)$
is an {\it $\al$-snowflaked version} of $(X,d_X)$.
\end{definition}

One of the standard metrics on diamonds $\{D_k\}_{k=1}^\infty$ is
the shortest path distance obtained under the assumption that each
edge in $D_k$ has length $\left(\frac12\right)^k$. Let us denote
this metric $d_{D_k}$.

\begin{proposition} \label{P:snowflake}
For any $\e\in (0,1/2)$ there exists  $\al\in (0,1)$ so that the natural identity bijections of (vertex sets of) weighted diamonds $\{W_k\}$
onto (vertex sets of) $\al$-snowflaked versions of diamonds
$\{(D_k,d_{D_k})\}$ have uniformly  bounded distortions for all $k\in \bbN$.
\end{proposition}

For the proof we need the following fact about the structure  of shortest paths in $W_n$, which was proved in \cite[Claim 4.1]{Ost14}.

\begin{lemma}[{\cite[Claim 4.1]{Ost14}}]\label{C:le2} A shortest path between two vertices in $W_n$ can contain edges of each possible
length: \[1,\left(\frac12+\ep\right), \left(\frac12+\ep\right)^2,
\left(\frac12+\ep\right)^3, \dots\] at most twice. Actually for
$1$ this can happen only once because there is only one such edge.
If there are two longest edges, they are adjacent.
\end{lemma}

\begin{proof} (This proof is a slightly modified version of the proof in \cite{Ost14};  we include it here  for the convenience of the readers.) Let $e$ be one of the longest edges
in the path and $\left(\frac12+\ep\right)^k$ be its length.  We
assume that   $k\ge 1$, the case  $k=0$ can be considered on the
same lines, it is even easier.

As for
diamonds, we define {\it weighted subdiamonds} to be subsets of $W_n$ which evolved from
an edge (as sets of vertices they coincide with the subdiamonds
defined in \cite{JS09,Ost14}). The edge from which a subdiamond
evolved is called its {\it diagonal}.

Consider the subdiamond $S$ containing $e$ with diagonal of length
$\left(\frac12+\ep\right)^{k-1}$ . Let $e=uv$. Without loss of
generality we may assume that $u$ is one of the ends of the
diagonal of $S$, denote the other end by $t$.
%\medskip

The rest of the path consists of two pieces, starting at $u$ and
$v$, respectively. We claim that the part which starts at $v$ can
never leave $S$. It obviously cannot leave through $u$. It cannot
leave through $t$, because otherwise the piece of the path between
$u$ and $t$ could be replaced by the diagonal of $S$, which is
strictly shorter.%\medskip

This implies that the part of the path in $S$ which starts at $v$
can contain edges only shorter than $\left(\frac12+\ep\right)^k$.
For the next edge in this part of the path we can repeat the
argument and get, by induction, that lengths of edges in the
remainder of the path in $S$ are strictly
decreasing.%\medskip

The part of the path which starts at $u$ can be considered
similarly. The last statement of the Lemma is immediate from the
proof.
\end{proof}

\begin{proof}[Proof of  Proposition~\ref{P:snowflake}] Let $\e\in (0,1/2)$ and let $\left(\frac12+\ep\right)$ be the weight of
edges of $W_1$ which are not in $W_0$. Pick $\alpha\in(0,1)$ so
that $\left(\frac12\right)^\alpha=\left(\frac12+\ep\right)$. Since
every edge of $D_k$ has length $\left(\frac12\right)^k$,  if we
raise its length to the power $\alpha$, we get the length of the
same edge in $W_k$. Therefore (since $d_{D_k}^\alpha$ satisfies
the triangle inequality) for any two vertices $x,y$ of $D_k$ (or
$W_k$) we have
$$(d_{D_k}(x,y))^\alpha\le d_{W_k}(x,y).$$

To get the inequality in the other direction, let $uv$ be the one
of the (at most two) longest edges in the shortest $xy$-path in
$W_k$. We claim that
\begin{equation}\label{claimDk}
d_{D_k}(x,y)\ge \frac12
d_{D_k}(u,v)
\end{equation}

If \eqref{claimDk} is satisfied, then by Lemma~\ref{C:le2} and
  since $uv$ is an edge we have
\[\begin{split}d_{W_k}(x,y)&\le
2d_{W_k}(u,v)\left(1+\left(\frac12+\ep\right)+\left(\frac12+\ep\right)^2+\dots\right)\\
&=\frac{2d_{W_k}(u,v)}{\frac12-\ep}=\frac{2}{\frac12-\ep}(d_{D_k}(u,v))^\al\le \frac{2^{\al+1}}{\frac12-\ep}(d_{D_k}(x,y))^\al\\
&=\frac{8}{1-4\e^2}(d_{D_k}(x,y))^\al.\end{split}
\]

We assume without loss of generality that the shortest $xy$-path
visits $u$ before $v$. To prove \eqref{claimDk}  we consider three
possible cases.

\begin{enumerate}

\item Both $x$ and $y$ are contained in the subdiamond $S$ with
diagonal $uv$.

\item Exactly one of $x$, $y$  is contained in the subdiamond $S$
with the diagonal $uv$.

\item None of $x$, $y$  is contained in the subdiamond $S$ with
the diagonal $uv$.

\end{enumerate}

Let $m_0\in \mathbb{N}$ be the smallest number such that
\begin{equation*}%\label{m0}
 \left(\frac12+\ep\right)+\left(\frac12+\ep\right)^2+\dots+\left(\frac12+
\ep\right)^{m_0}\ge
1+\left(\frac12+\ep\right)^{m_0}.
\end{equation*}
It is clear that $m_0\ge 3$ if $\ep\in(0,\frac12)$.

\begin{figure}[h]
\centering
\includegraphics{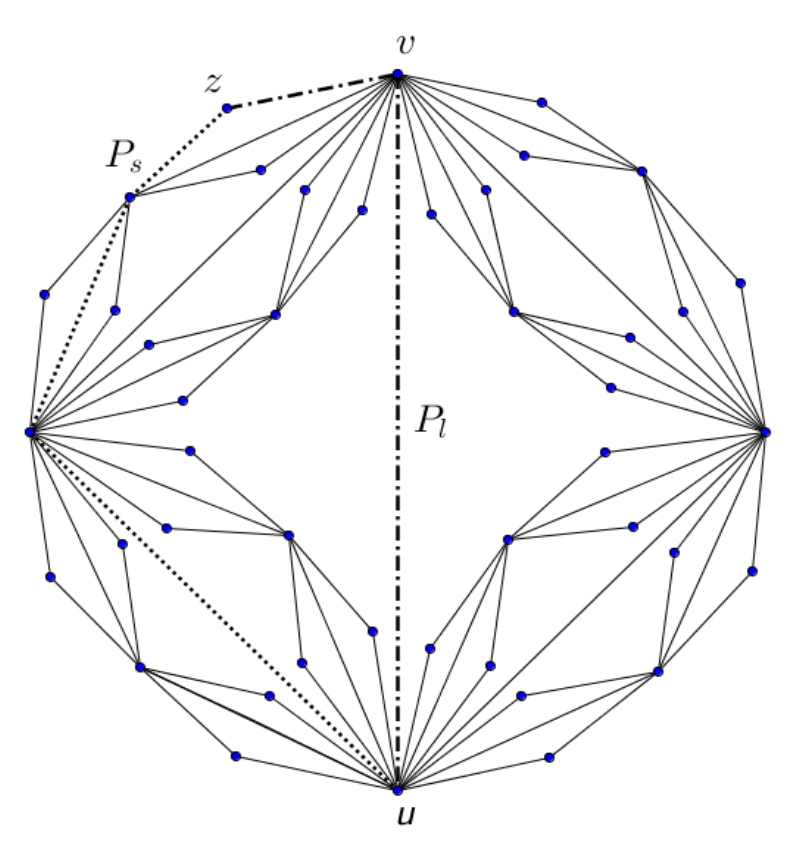}
\caption{Paths $P_s$ and $P_l$ in subdiamond
$S$}\label{F:SubdiamondS}
\end{figure}

We show that in Case 1 $y$ is contained in one of the subdiamonds
of $S$ whose diagonal has length
$\le\left(\frac12+\ep\right)^{m_0-1}d_{W_k}(u,v)$, and $v$ is one
of its ends. In fact, otherwise the $vy$-path in $W_k$, which is a
part of the shortest $xy$-path which we consider, contains a
vertex $z$ of the subdiamond $S$ for which $zv$ is an edge of
$W_k$ satisfying
\[d_{W_k}(z,v)=\left(\frac12+\ep\right)^{t}d_{W_k}(u,v)\]
for some positive integer $t\le m_0-1$.  But then there is a
$uz$-path $P_s$ in $W_k$ which is strictly shorter than any
$uz$-path $P_l$ through $v$, contrary to our assumption that one
of the shortest $xy$-paths passes through $u$ and $v$. To see this
we observe that we can pick $P_s$ (see Figure \ref{F:SubdiamondS})
of length
\[\left(\left(\frac12+\ep\right)+\left(\frac12+\ep\right)^2+\dots+\left(\frac12+\ep\right)^{t}\right)d_{W_k}(u,v),\]
and that $P_l$ is of length at least
\[\left(1+\left(\frac12+\ep\right)^{t}\right)d_{W_k}(u,v),\]
the conclusion $\hbox{length}(P_s)<\hbox{length}(P_l)$ follows
from $t\le m_0-1$ and the definition of $m_0$.

The statement that  $y$ is contained in one of the subdiamonds of
$S$ whose diagonal has length
$\le\left(\frac12+\ep\right)^{m_0-1}d_{W_k}(u,v)$, and $v$ is one
of its ends implies that
$d_{D_k}(y,v)\le\left(\frac12\right)^{m_0-1}d_{D_k}(u,v)\le\frac14d_{D_k}(u,v)$
since $m_0\ge 3$. Similarly we prove that $d_{D_k}(x,u)\le\frac14
d_{D_k}(u,v)$. We conclude that
\[d_{D_k}(x,y)\ge d_{D_k}(u,v)-d_{D_k}(y,v)-d_{D_k}(x,u)\ge
\left(1-\frac24\right)d_{D_k}(u,v)=\frac12d_{D_k}(u,v),\] so the
inequality \eqref{claimDk} is satisfied when Case 1 holds.

In Cases 2 and 3 we consider the subdiamond of $W_k$ with the
diagonal of the length $\left(\frac12+\ep\right)^{m-1}$ (where
$\left(\frac12+\ep\right)^{m}$ is the length of $uv$ in $W_k$). We
may assume without loss of generality that the diagonal of this
subdiamond is of the form $uw$. We denote by $\widetilde v$ the
other vertex for which both $\widetilde vu$ and $\widetilde vw$
have length $\left(\frac12+\ep\right)^{m}$.

Because of lack of symmetry we consider separately Case 2a, where
$x$ is and $y$ is not in the subdiamond $S$ with the diagonal
$uv$, and Case 2b, where $y$ is and $x$ is not in the subdiamond
$S$ with the diagonal $uv$.

In Case 2a the vertex $y$ is contained in the subdiamond with the
diagonal $vw$ (because the edge $uw$ is shorter than any other
$uw$-path). Also $d_{D_k}(x,u)\le\frac14d_{D_k}(u,v)$ for the same
reason as in Case 1. Hence
$d_{D_k}(x,y)>\frac34d_{D_k}(u,v)$.

In Case 2b there are two subcases: where $x$ is in the subdiamond
with the diagonal $uw$, and where $x$ is not there. In both cases
it is easy to see that the shortest in $D_k$ $xy$-path cannot be
shorter than $d_{D_k}(u,y)\ge \frac34d_{D_k}(u,v)$.

In Case 3 the only situation in which the desired inequality is
not immediate is the situation where $x$ is in the subdiamond with
the diagonal $\widetilde vw$ and $y$ is in the subdiamond with the
diagonal $vw$. In this case the shortest path contains both the
edge $\widetilde vu$ and $uv$. Replacing these edges by the edges
$\widetilde vw$ and $wv$, we get   one of the previously
considered cases.
\end{proof}

In this context it is natural to recall the following
well-known result of Assouad \cite{Ass83} (see also \cite[Chapter
12]{Hei01}).

\begin{definition}\label{D:Doubl} A metric space is called {\it
doubling} if there exists $N<\infty$ such that each ball in this
space can be covered by at most $N$ balls of twice smaller radius.
\end{definition}

\begin{theorem}[\cite{Ass83}]\label{T:Assouad} Each snowflaked version of a
doubling metric space admits a bilipschitz embedding into a
Euclidean space.
\end{theorem}

\begin{remark}\label{dimension}
In the original proof of Theorem~\ref{T:Assouad} the dimension $N$ of the receiving Euclidean space and the distortion of the embedding depend both on the doubling constant of the  metric space and on the amount $\al$ of snowflaking, with $N$ going to $\infty$, as $\al$ approaches 1. Recently, Naor and Neiman \cite{NN12} (cf. also \cite{DS13}) obtained estimates of $N$ depending only on the doubling constant and  independent of $\al$, for $\al\in (1/2,1)$.
\end{remark}

\begin{remark}\label{R:Ass83&Ost14}
 The spaces $\{W_k\}$ do not satisfy the assumptions of Theorem~\ref{T:Assouad}, i.e.
the spaces
$\{D_k\}_{k=1}^\infty$ are not uniformly doubling. Indeed,
  balls of radius $\left(\frac12\right)^k$ centered at the
bottom of $D_k$  contain
$2^k$ vertices of mutual distance $\left(\frac12\right)^{k-1}$, in addition to the bottom vertex, and thus  no pair
of such vertices is contained in any ball of radius
$\left(\frac12\right)^{k+1}$.

 Note that the doubling condition is important for
$\ell_2$-embeddability in Theorem~\ref{T:Assouad}. Consider, for
example, the space $L_1(0,1)$ and
$\alpha\in\left(\frac12,1\right)$, and apply \cite[Lemma
2.1]{AB14}. Thus the results about embeddings of $\{W_k\}$ in
\cite{Ost14} and in the present paper are not covered by the
Assouad theorem (Theorem \ref{T:Assouad}).
\end{remark}

\subsection{Weighted diamonds are not included in the set of examples
presented in Section \ref{S:Equi&Eucl&Mix}}\label{S:Different}

The goal of this section is to show that the bilipschitz
embeddability of $W_n$'s into an arbitrary Banach space of
dimension $\exp(\Omega((\log\log |W_n|)^2))$ with uniformly
bounded distortion does not follow from
results of Section~\ref{S:Equi&Eucl&Mix}.

We start from two simple results about nonembeddability of $W_n$'s
with uniformly bounded distortion into low-dimen\-sio\-nal
Euclidean spaces and equilateral spaces.

\begin{proposition}\label{P:EucNoLog}
The distortions of embeddings of $W_n$ into $\ell_2^{k(n)}$ can be
uniformly boun\-ded only if $k(n)\ge cn\approx c\log(|W_n|)$ for
some $c>0$.
\end{proposition}

\begin{proof} This is an immediate consequence of Lemma~\ref{L:lowerequi} and the following lemma.
\end{proof}

\begin{lemma}\label{L:equiwn}
The spaces $W_n$ contain equilateral subsets of sizes
$2^m$ for all $m\le n$.
\end{lemma}

\begin{proof} The bottom vertex
$b$ of $W_n$ is adjacent in $W_m\subseteq W_n$ to $2^m$ vertices $\{a_j\}_{j=1}^{2^m}$
with edges of length $\left(\frac12+\ep\right)^m$ joining them
and $b$, and thus for all $1\le i,j\le 2^m$, with $i\ne j$,
$d_{W_n}(a_i,a_j)=2\left(\frac12+\ep\right)^{m}$.
\end{proof}

\begin{remark} The same argument shows non-embeddability of $W_n$'s with uniformly
bounded distortion into any low-dimensional Banach spaces.
\end{remark}

\begin{proposition}\lb{nonembequilateral}
The spaces $W_n$  cannot be embedded with uniformly bounded
distortion into any equilateral spaces.
\end{proposition}
\begin{proof}
The maximal distance between two elements in $W_n$ is $\ge 1$, and the
minimal distance is $\left(\frac12+\ep\right)^n$. Hence any
embedding of $W_n$ into an equilateral space has distortion  greater than or equal
to $\left(\frac12+\ep\right)^{-n}$. Thus, since $\e\in(0, \frac12)$,
 the distortions are not uniformly bounded.
\end{proof}

In the next two propositions we show  that $W_n$'s do not admit
bilipschitz embeddings with uniformly bounded distortions into
spaces of the form $M_\be(\N)$, where $\be\ge 1/2$, the collection
$\N$ of metric spaces is such that $W_n$ do no admit bilipschitz
embeddings with uniformly bounded distortions into any $N\in\N$
and $M$ is either an equilateral space or a metric space that
admits a bounded-distortion embedding into a $O((\log\log
|W_n|)^2)$-dimensional Euclidean space.

\begin{proposition}\label{notequibeta}
For $n\in\bbN$, let $A_n, B_n>0$ be  constants so that there
exists a finite equilateral metric space $M$, $\be\ge 1/2$, and
$\N=\{N_x\}_{x\in M}$
 a collection of finite metric spaces so that for any $x\in M$, $W_n$ cannot be embedded into $N_x$ with distortion $\le A_n/B_n$, and so that there exists an embedding $\phi:W_n\lra M_\be(\N)$ such that for all $u_1, u_2\in W_n$,
\begin{equation}\label{lipequilateral}
B_n d_{W_n}(u_1,u_2)\le d_\be(\phi(u_1),\phi(u_2))\le A_n d_{W_n}(u_1,u_2).
\end{equation}
Then
\[\lim_{n\to\infty}\frac{A_n}{B_n}=\infty.\]
\end{proposition}

\begin{proof}
Since $M$ is equilateral, we have
\[d_\be(y_1,y_2)=\begin{cases} d_{N_x}(y_1,y_2) \ \ \ \ &{\text{\rm if }} y_1,y_2\in N_x;\\
\be\g &{\text{\rm if }} y_1\in N_{x_1}, y_2\in N_{x_2}, x_1\ne x_2
\end{cases}\]
where $\g=\max_{x\in M} \diam(N_x)$.

Since for any $x\in M$, $W_n$ cannot be embedded into $N_x$ with distortion $\le A_n/B_n$, there exist $v_1, v_2\in W_n$ so that
\begin{equation}\label{diffx}
\phi(v_1)\in N_{x_1},\phi(v_2)\in N_{x_2},
\text{\rm and } x_1\ne x_2.
\end{equation}

 In $W_n$ we can travel between any pair of vertices
using  edges of the smallest
available length, i.e. there exists a sequence of vertices $\{u_j\}_{j=0}^{j_0}$ so that $u_0=v_1$, $u_{j_0}=v_2$, and for all $0<j\le j_0$,
$$d_{W_n}(u_{j},u_{j-1})=\left(\frac12 +\e\right)^n.$$
 For $0<j\le j_0$, let $z_j=\phi(u_j)$. By \eqref{diffx}, there exists $0<i\le j_0$ so that $z_{i-1}\in N_{x}$, $z_{i}\in N_{x'}$, and $x\ne x'$. Thus $d_\be(z_{i-1},z_i)=\be\g$. Hence by \eqref{lipequilateral}, we have
\begin{equation}\label{gamma}
B_n \left(\frac12 +\e\right)^n\le \be\g\le A_n\left(\frac12 +\e\right)^n.
\end{equation}

Now let $w_1, w_2\in W_k$ be such that $d_{W_k}(w_1,w_2)=1$.
Then
\begin{equation}\label{diameter}
B_n \le d_\be(\phi(w_1),\phi(w_2))\le A_n.
\end{equation}

If there exists $x\in M$ so that $\phi(w_1),\phi(w_2)\in N_x$,
then
$d_\be(\phi(w_1),\phi(w_2))=d_{N_x}(\phi(w_1),\phi(w_2))\le\g$.
Therefore, combining \eqref{gamma} and \eqref{diameter}, since
$\be\ge1/2$, we obtain
\[\frac{A_n}{B_n}\ge{\be}\left(\frac{1}{\frac12 +\e}\right)^n\ge \frac12\left(\frac{1}{\frac12 +\e}\right)^n.\]

If  $\phi(w_1)\in N_{x_1},\phi(w_2)\in N_{x_2}$, where $x_1\ne
x_2$, then $d_\be(\phi(w_1),\phi(w_2))=\be\g$.  Therefore,
combining \eqref{gamma} and \eqref{diameter}, we obtain
\[\frac{A_n}{B_n}\ge \left(\frac{1}{\frac12 +\e}\right)^n.\]
Since $\left(\frac12 +\e\right)<1$, in either case the proposition
is proven.
\end{proof}

\begin{proposition}\label{P:NEWithEuclStart}
Let $C\ge 1$. For $n\in\bbN$, let $A_n, B_n>0$ be  constants,
possibly depending on $C$, so that there exists a finite  metric
space $M$ which admits a $C$-embedding  into a
$O((\log\log|W_n|)^2)$-dimensional Euclidean space, $\be\ge 1/2$,
and $\N=\{N_x\}_{x\in M}$
 a collection of finite metric spaces so that for any $x\in M$, $W_n$ cannot be embedded into $N_x$ with distortion $\le A_n/B_n$, and so that there exists an embedding $\phi:W_n\lra M_\be(\N)$ such that for all $u_1, u_2\in W_n$,
\begin{equation}\label{liphilbert}
B_n d_{W_n}(u_1,u_2)\le d_\be(\phi(u_1),\phi(u_2))\le A_n d_{W_n}(u_1,u_2).
\end{equation}
Then for all $C\ge 1$,
\[\lim_{n\to\infty}\frac{A_n}{B_n}=\infty.\]
\end{proposition}

\begin{proof}
By Lemma~\ref{L:equiwn}, since $\lfloor \sqrt{n}\rfloor\le n$, the
spaces $W_n$ contain equilateral subsets $S_n$ of sizes
$2^{\lfloor \sqrt{n}\rfloor}$ with distances between their
elements equal to $2\left(\frac12+\ep\right)^{\lfloor
\sqrt{n}\rfloor}$. By \eqref{cardwn} we have ${\lim_{n\to
\infty}\frac{\lfloor
\sqrt{n}\rfloor}{(\log_2\log_2|W_n|)^2}=\infty}$. Thus, by
Lemma~\ref{L:lowerequi}, $S_n$ do not admit bilipschitz embeddings
with uniformly bounded distortions into a
$O((\log\log|S_n|)^2)=O((\log\log|W_n|)^2)$-dimensional Euclidean
space, and thus into $M$. Therefore $\phi$ maps some elements of
$u_1,u_2\in S_n$   into the same set $N_{x_0}$ for some $x_0\in M$
(see Definition~\ref{D:MetrComp}). We have
\begin{equation*}\label{liphilbertsame}
B_n 2\left(\frac12+\ep\right)^{\lfloor \sqrt{n}\rfloor}\le
d_{N_{x_0}}(\phi(u_1),\phi(u_2))\le A_n
2\left(\frac12+\ep\right)^{\lfloor \sqrt{n}\rfloor}
\end{equation*}
and
\begin{equation}\label{lipdiam}
\max_{x\in M}\diam(N_x)\ge \diam(N_{x_0})\ge d_{N_{x_0}}(\phi(u_1),\phi(u_2))\ge
B_n 2\left(\frac12+\ep\right)^{\lfloor \sqrt{n}\rfloor}.
\end{equation}

On the other hand,
our assumptions imply that all elements of $W_n$ are not mapped into
the same set $N_x$, that is
 there exist vertices $v_1$ and $v_2$ of $W_n$ which are mapped into sets
$N_{x_1}$ and $N_{x_2}$ with $x_1\ne x_2$. Using the same  argument as in the proof of \eqref{gamma} in Proposition~\ref{notequibeta}, we obtain that there exist $y_1\ne y_2\in M$ so that
\begin{equation}\label{gamma2}
B_n \left(\frac12 +\e\right)^n\le \be\g d_M(y_1,y_2)\le A_n\left(\frac12 +\e\right)^n,
\end{equation}
where $\g=\frac{\max_{x\in M} \diam(N_x)}{\min_{x\ne y\in M}d_M(x,y)}$. By \eqref{lipdiam} and \eqref{gamma2} we get
\[
\begin{split}
A_n\left(\frac12 +\e\right)^n&\ge \be\g d_M(y_1,y_2)= \be \cdot \frac{\max_{x\in M} \diam(N_x)}{\min_{x\ne y\in M}d_M(x,y)}\
d_M(y_1,y_2)\\
&\ge \be \max_{x\in M} \diam(N_x) \ge \be B_n 2\left(\frac12+\ep\right)^{\lfloor \sqrt{n}\rfloor}.
\end{split}
\]

Therefore

\[\frac{A_n}{B_n}\ge\left(\frac12+\ep\right)^{\lfloor \sqrt{n}\rfloor-n}.
\]
Since $\left(\frac12 +\e\right)<1$, the proposition is proven.
\end{proof}

We are now ready to prove the main result of this subsection.

We denote by $\mathcal{E}$ the class of all finite equilateral
spaces, by ${\mathcal{L}}_{n,C}$, for $n\in\bbN$ and $C\ge 1$, the
class of all finite metric spaces that admit  embeddings into
$O((\log n)^2)=O((\log\log |W_n|)^2)$-dimensional Euclidean spaces
with distortion $\le C$, and let
${\mathcal{M}}_{n,C}={\mathcal{E}}\cup{\mathcal{L}}_{n,C}$.

\begin{theorem}\lb{Wndoesnotembed}
For any  $C, \beta\ge 1$, the spaces $W_n$ do not admit embeddings with a uniformly  bounded distortion  into  metric spaces $V\in\comp_{\be}({\mathcal{M}}_{n,C})$.
\end{theorem}

\begin{proof} The proof is by induction on the level $m$ of complexity of  spaces  $V\in\comp_{\be}({\mathcal{M}}_{n,C})$ as defined in \eqref{complexity}.
The base case $m=0$ follows from Propositions~\ref{P:EucNoLog} and
\ref{nonembequilateral}. By Proposition~\ref{P:trans-comp}, the
inductive step follows from Propositions~\ref{notequibeta} and
\ref{P:NEWithEuclStart}.
\end{proof}

\subsection{Embeddings of weighted diamonds  into
low dimensional Banach spaces with uniformly bounded
distortion}\label{S:EmbW_kAnyBS}

We prove first that weighted diamonds can be embedded  into
low dimensional spaces with a basis with  distortion which is bounded by a constant that depends only on $\e$ and the basis constant of the target space.

\begin{theorem}\label{T:EmbAny} For every $C\ge 1$ and $\e\in(0,1/2)$ there exists a constant $D\ge 1$ so that for every $n\ge 2$,   $W_n$ $D$-embeds into every Banach space $X$ with
a Schauder basis with basis constant smaller than or equal to $C$ and of dimension $\ge\frac12 (\log_2 |W_n|)^2$.
\end{theorem}

To apply Theorem \ref{T:EmbAny} to embeddings into arbitrary
finite dimensional spaces we need to know what is the best
estimate for the dimension of a subspace with a basis constant $C$
in any $n$-dimensional  Banach spaces.  More precisely, we are
interested in lower bounds  for the following problem.

\begin{problem}\label{P:Basis}
Let $C\in(1,\infty)$. Define the function $f_C(n)$ to be the
largest $k\in\mathbb{N}$ so that each $n$-dimensional Banach space
contains a $k$-dimensional subspace with basis constant at most
$C$. What are the estimates for  $f_C(n)$?
\end{problem}

Known upper estimates can be found in  \cite{MT93}. Many   experts believe that the techniques of \cite{MT93}
(which go back to Gluskin \cite{Glu81} and Szarek \cite{Sza83})
can be used to achieve the upper bound of order $n^{1/2}$ (perhaps
multiplied by some power of a logarithm), but it does not seem
that anyone has worked this out.

 The best lower bound for $f_C(n)$ that we have found in the literature is
a result Szarek and Tomczak-Jaegermann \cite{ST09}, where they
studied the nontrivial projection problem. Thus they were
interested in `large' subspaces with `large' codimension which
have small projection constants in comparison with their
dimension, but since the subspaces found in \cite{ST09} were close
to $\ell_p^k$ with $p\in\{1,2,\infty\}$, their result can be used
for our purposes. It appears that techniques of
\cite{AM83,Rud95,ST09,Tal95} could be useful for further work on
lower estimates for Problem \ref{P:Basis}.

We state here the result of \cite{ST09} in the form closest to the
answer to Problem~\ref{P:Basis}.

\begin{theorem}[\cite{ST09}] \label{SzTJ}
There exist absolute constants $A, B, C>0$ so that for every $n\ge
A$ and for every $n$-dimensional normed space $X$, there exists a
subspace $Y\subseteq X$ so that $\dim Y\ge B\exp(\frac12\sqrt{\ln n})$
and $Y$ is $C$-isomorphic to an $\ell_p$-space for some
$p\in\{1,2,\infty\}$.
\end{theorem}

As an immediate consequence of Theorems~\ref{T:EmbAny} and
\ref{SzTJ} we obtain that $W_n$'s can be embedded with uniformly bounded distortion in an  arbitrary Banach space of  dimension   $\exp(c(\log \log |W_n|)^2)$, for some  fixed $c>1$.

\begin{corollary}\label{cor:anyemb-old}
For every $\e\in (1/2,1)$, there exist  constants  $c, C>1$ so that for every $n\ge C$, $W_n$ can be embedded in every Banach
space $X$ with $\dim X\ge\exp(c(\log \log |W_n|)^2)$ with the distortion
bounded from above by a constant which depends only on
$\ep$.
\end{corollary}

The remainder of this section is devoted to the proof Theorem~\ref{T:EmbAny}.

\begin{proof}[Proof of Theorem \ref{T:EmbAny}]
Fix $C\ge 1$, $\e>0$,  and $n\ge 2$.
Let  $X$ be any Banach space  with $\dim X=d\ge\frac 12 (\log_2|W_n|)^2$, and with a
Schau\-der basis $\{x_i\}_{i=0}^{d-1}$ with basis constant at
most $C$ so that $\|x_i\|=1$ for all $i$. Let $Y_0=\Span\{x_0\}$,
$Y_1=\Span\{x_1\}$, $Y_m=\Span\{x_j :
\left(\sum_{k=1}^{m-1}k\right) +1 \le j \le \sum_{k=1}^{m}k\}$,
for $m=2,\dots,n$. Note that, for $m=1,\dots,n$, $\dim Y_m=m$, and thus, by \eqref{cardwn} and since $n\ge 2$,
$$ 1+\sum_{k=1}^{n}k= 1+\frac{n(n+1)}{2}\le\frac{(2n-1)^2}{2}\le \frac12 (\log_2|W_n|)^2\le d.$$
Thus there are enough    basis
 vectors in $X$ to define all these subspaces.
For  $m=0,\dots,n$, let $\{y_{m,k}\}_{k=1}^{2\cdot 4^{m-1}}$ be elements of the unit
sphere of $Y_m$ satisfying the conditions of
Lemma~\ref{L:apart} with $\de=1/16$. It is easy to see that for any $m=0,\dots,n$, any $1\le k\ne k'\le 2\cdot 4^{m-1}$, and any $a$ with $0\le a\le 1$, we have
\begin{equation}\label{moreover}
\|y_{m,k}-ay_{m,k'}\|\ge \de/2.
\end{equation}

 Note that for $m> m'$, $y_{m,k}$ and $y_{m',k'}$ are supported on disjoint intervals with respect to the basis
 $\{x_i\}_{i=0}^{d-1}$, and therefore
\begin{equation}\label{big}
\|y_{m,k}-y_{m',k'}\|\ge\frac1{C}\|y_{m',k'}\|=\frac1{C}.
\end{equation}

We construct an embedding $S_n:W_n\to X$ in the following
way.

\begin{itemize}

\item The map $S_n$ maps the vertices of $D_0$ to $0$ and $x_0$,
respectively. It is clear that $S_n|_{W_0}$ is an isometric
embedding.

\item The map $S_{n}|_{W_m}$, $1\le m\le n$, is an extension of
the map $S_{n}|_{W_{m-1}}$. Note that for each $m\ge 1$,
$|W_m\setminus W_{m-1}|=2\cdot 4^{m-1}$. Let
$\s_m:W_m\setminus W_{m-1}\to\{1,\dots,2\cdot 4^{m-1}\}$ be any
bijective map.
 Each vertex $w\in W_m\backslash W_{m-1}$ corresponds to a pair of vertices of
$W_{m-1}$: $w$ is the vertex of a $2$-edge path joining $u$ and
$v$. We map the vertex $w$ to
\begin{equation}\label{E:DefS_n}
\frac12(S_{n}u+S_n v)+\e\left(\frac12+\e\right)^{m-1}y_{m,\s_m(w)}.
\end{equation}
%\medskip

\end{itemize}

Now we estimate the distortion of $S_n$. First we observe that the
map $S_n$ is $1$-Lipschitz. This can be proved for $S_n|_{W_m}$ by
induction on $m=0,1,\dots,n$. It suffices to observe that for each
edge $uv$ in $W_{m-1}$ and each vertex $w$ satisfying the
condition of the previous paragraph we have
$d_{W_n}(u,w)=\left(\frac12+\ep\right)^{m}$ and
$$||S_nu-S_nw||\le
\frac12||S_nu-S_nv||+\ep\left(\frac12+\ep\right)^{m-1}\le\left(\frac12+\ep\right)^{m}.$$%\medskip

To estimate from above the Lipschitz constant of $S_n^{-1}$ we
consider any shortest path $P$ between two vertices $w,z$ in
$W_n$. Let $\left(\frac12+\ep\right)^t$ be the length of the
longest edge in it. By Lemma~\ref{C:le2},

\begin{equation}\label{E:dAbove}
d_{W_n}(w,z)\le\frac{2\left(\frac12+\ep\right)^t}{\left(\frac12-\ep\right)}.
\end{equation}

On the other hand, since the subspaces $Y_m$ are supported on
disjoint intervals with respect to the basis
$\{x_i\}_{i=0}^{d-1}$, for every $m\in\{1,\dots,n\}$, we have
\begin{equation}\label{E:FunctBelow}
||S_nw-S_nz||\ge \frac1{2C} \|(S_nw-S_nz)\big|_{Y_m}\|
\end{equation}
 where  $x|_{Y_m}$ denotes the
natural projection onto $Y_m$.%\medskip

Let $m_0\in
\mathbb{N}$ be the smallest number  such that
\begin{equation}\label{m0} \left(\frac12+\ep\right)+\left(\frac12+\ep\right)^2+\dots+\left(\frac12+
\ep\right)^{m_0}> 1+\left(\frac12+\ep\right)^{m_0}.\end{equation}
Such a number $m_0$ obviously exists if $\ep>0$. It is clear that
$m_0\ge 3$ if $\ep<\frac12$, and that $m_0$ depends only on $\ep$.
%\medskip

\remove{ Note that if $xy$ is an  edge of length
$\left(\frac12+\ep\right)^t$, then
  \begin{equation}\label{oneedge}
\|S_nx-S_ny\|\ge
\frac1{2C}\e\left(\frac12+\ep\right)^{t-1}.
\end{equation}

Indeed, this is obvious if $t=0$. If $t\ge 1$, one of the
vertices, say $y$, is of a later generation than the other vertex.
Then $y\in W_t\setminus W_{t-1}$, $x\in W_{t-1}$, and there exists
a vertex $\bar x\in W_{t-1}$ so that $xy$ is an edge in a
subdiamond with diagonal $\bar xy$, which is of length
$\left(\frac12+\ep\right)^{t-1}$ and
$$S_ny=\frac12(S_nx+S_n\bar x)+\e\left(\frac12+\ep\right)^{t-1}y_{t,\s_t(y)},$$
where $S_nx, S_n\bar x\in\Span\{\bigcup_{m<t} Y_m\}$ and
$y_{t,\s_t(y)}\in Y_t$. Thus \[\begin{split}\|S_nx-S_ny\|&\ge
\frac1{2C} \left\|(S_nx-S_ny)|_{Y_m}\right\|\\&= \frac1{2C}
\left\|\e\left(\frac12+\ep\right)^{t-1}y_{t,\s_t(y)}\right\|\\&=\frac1{2C}\e\left(\frac12+\ep\right)^{t-1}.\end{split}\]}

Now we turn to estimates of $||S_nw-S_nz||$ from below. Let $xy$
be one of the edges of the largest length
$\left(\frac12+\ep\right)^t$ in the path $P$ from $w$ to $z$ (by
Lemma~\ref{C:le2} we know that path $P$ contains at most two such
edges; and that if there are two of them, they share a vertex). We
restrict our attention to the case where at least one of the
vertices $x,y$ is not it $W_0$; the excluded case can be
considered along the same lines. \Buo we assume that $y\in
W_t\setminus W_{t-1}$ and $x\in W_{t-1}$. Let $\bar x$ be the
vertex in $W_{t-1}$ so that $x\bar x$ is an edge of the length
$\left(\frac12+\ep\right)^{t-1}$ and $y$ belongs to the subdiamond
with diagonal $x\bar{x}$. We assume that our notation is chosen in
such a way that $z$ is closer to $y$ than to $x$. Then the part of
the path $P$ from $y$ to $z$ does not contain an edge of length
$\left(\frac12+\ep\right)^t$, and $z$ is either in the subdiamond
with diagonal $yx$, or in the subdiamond with diagonal $y\bar{x}$.
%\medskip

To  simplify the notation, let us denote the vector
$\ep\left(\frac12+\ep\right)^{t-1}y_{t,\sigma_t(y)}$ by
$\pi_{t,y}$.

\begin{lemma}\label{L:Try} {\bf (i)} If $z$ is in the subdiamond with the
diagonal $y\bar{x}$, then \[S_nz|_{Y_t}=\varrho_1\pi_{t,y}\] for
some $\varrho_1\ge\left(\frac12\right)^{m_0-1}$.

{\bf (ii)} If $z$ is in the subdiamond with diagonal $yx$, then
\[(S_ny-S_nz)|_{Y_t}=\varrho_2\pi_{t,y}\] for some
$0\le\varrho_2\le\left(\frac12\right)^{m_0-1}$.

{\bf (iii)} If $w$ is in the subdiamond with diagonal $yx$, then
\[S_nw|_{Y_t}=\varrho_3\pi_{t,y}\] for some
$0\le\varrho_3\le\left(\frac12\right)^{m_0-1}$.

{\bf (iv)} If $w$ is not in the subdiamond with the diagonal $yx$,
then
$$S_nw|_{Y_t}=\varrho_4 y_{t,k},$$
 for some
$k\ne\sigma_t(y)$ and $\varrho_4\in[0,1]$.

\end{lemma}

\begin{proof} {\bf (i)} Let $z$ is in the subdiamond with diagonal $y\bar
x$. Observe that ends of edges of length
$\le\left(\frac12+\ep\right)^{m_0+t-1}$ with one end at $\bar{x}$
and the other end in the subdiamond with the diagonal $\bar{x}y$
cannot be in $P$ because then, by \eqref{m0}, the path through
$\bar{x}$ would be shorter. Therefore \remove{

the coefficient of $\pi_{t,y}$ in the projection of $S_nz$ on
$Y_t$ is at least $\left(\frac12\right)^{m_0-1}$.

If , then, by \eqref{m0},}
\begin{equation*}%\label{Snwp}
S_n z=(1-b)S_n\bar{x} +bS_ny + \bar z_t,
\end{equation*}
where  $S_ny \in B_{t}\DEF\Span\{x_j: j\le\sum_{k=1}^t
k\}=\Span(\bigcup_{m=1}^t Y_t)$, $S_n\bar x\in B_{t-1}$,  $\bar
z_t\in T_t\DEF\Span\{x_j: j>\sum_{k=1}^t k\}$, and $b\ge
\left(\frac12\right)^{m_0-1}$. Note that
\begin{equation}\label{Sny}
S_ny=\frac12(S_nx+S_n\bar x)+\pi_{t,y},
\end{equation}
where  $S_nx, S_n\bar x \in B_{t-1}$ and $\pi_{t,y}\in Y_t$. Hence
\begin{equation*}%\label{Snw}
S_nz|_{Y_t}=b\pi_{t,y},
\end{equation*}
the conclusion follows.
\medskip

\noindent{\bf (ii)} If  $z$ is in the subdiamond with diagonal
$yx$, since $yx$ is a part of a shortest path, we conclude that
 the longest edge in the part of the path $P$ from $y$ to $z$ has length
$\le\left(\frac12+\ep\right)^{t+m_0}$, where $m_0$ satisfies
\eqref{m0}. Thus
\begin{equation}\label{Snz1}
S_n z= aS_ny +(1-a)S_nx + z_t,
\end{equation}
where $S_ny \in B_t$, $S_nx \in B_{t-1}$,  $z_t\in T_t$, and
\begin{equation*}%\label{heighta}
1\ge a \ge 1- \sum_{k=m_0}^\infty\left(\frac12\right)^{k}=
1-\left(\frac12\right)^{m_0-1}.
\end{equation*}

By \eqref{Sny} we get
\[(S_ny-S_nz)|_{Y_t}=\pi_{t,y}-a\pi_{t,y}=(1-a)\pi_{t,y},\]
the conclusion follows.\medskip

\noindent{\bf (iii)} If $w$ is in the subdiamond for which $xy$ is
the diagonal, similarly as in \eqref{Snz1}, we obtain
$$S_n w= cS_nx +(1-c)S_ny + w_t,$$
where $S_nx\in B_{t-1}, S_ny \in B_t$, $w_t\in T_t$, and $ 1\ge
c\ge  1-\left(\frac12\right)^{m_0-1}$. By \eqref{Sny},
\[S_nw|_{Y_t}=(1-c)\pi_{t,y},\]
and we are done in this case.

\noindent{\bf (iv)} If $w$ is not in the subdiamond for which $xy$
is the diagonal, let $q\in W_t\setminus W_{t-1}$, $q\ne y$, be the
vertex which is an endpoint of an edge of length
$\left(\frac12+\ep\right)^{t}$ which is a diagonal of the
subdiamond that contains $w$. By construction, the projection of
$S_nw$ onto the subspace $Y_t$ is a multiple of
$y_{t,\sigma_t(q)}\ne y_{t,\sigma_t(y)}$,  with some coefficient
$\varrho_4\in[0,1]$.\end{proof}

\remove{ Thus, by \eqref{Snz2} and Lemma~\ref{L:apart},
\begin{equation*}%\label{big2}
\begin{split}
\|S_nz-S_nw\|&\ge \frac1{2C} \|(S_nw-S_nz)|_{Y_t}\|\\
 &= \frac1{2C} \left\|\g\e\left(\frac12+\e\right)^{t-1} y_{t,\s_t(y)}-\al y_{t,\s_t(q)}\right\|\\
&\ge
\frac1{2C}\frac{\de}2\max\left(\g\e\left(\frac12+\e\right)^{t-1},\al\right)
\\&\ge  \frac1{4C}\e\de\left(\frac12+\e\right)^{t-1}\g_1(\e),
\end{split}
\end{equation*}
which finishes the proof.}

Observe that Lemma \ref{L:Try} implies the estimate for the
Lipschitz constant of $S_n^{-1}$, and thus Theorem \ref{T:EmbAny},
in all of the cases except the case where both {\bf (i)} and {\bf
(iii)} hold. Consider, for example the case where {\bf (i)} and
{\bf (iv)} hold. Then (we use \eqref{E:dAbove},
\eqref{E:FunctBelow}, the conclusions of {\bf (i)} and {\bf (iv)},
the definition of $\pi_{t,y}$ and \eqref{moreover})
\[\begin{split}
\frac{d_{W_n}(w,z)}{||S_nw-S_nz||}&\le\frac{2\left(\frac12+\ep\right)^t}{\left(\frac12-\ep\right)\frac1{2C}||\varrho_1\pi_{t,y}-\varrho_4
y_{t,k}||}\\
&\le\frac{4C\left(\frac12+\ep\right)^t}{\left(\frac12-\ep\right)\frac{\delta}2\varrho_1\ep\left(\frac12+\ep\right)^{t-1}}\le
\frac{8C\left(\frac12+\ep\right)}{\left(\frac12-\ep\right)\delta\ep\left(\frac12\right)^{m_0-1}},
\end{split}\]
and this number depends only on $C$ and $\ep$.\medskip

 It remains to consider the  case when both {\bf (i)} and {\bf
(iii)} hold. In this case we estimate from below the norm of
$(S_nw-S_nz)|_{Y_{t-1}}$. We use $S_n z=(1-b)S_n\bar{x} +bS_ny +
\bar z_t$ and $S_n w= cS_nx +(1-c)S_ny + w_t$ with $1\ge b\ge
\left(\frac12\right)^{m_0-1}$ and $1\ge c \ge
1-\left(\frac12\right)^{m_0-1}$. The value of $b$ actually does
not matter for our argument, it is only important that $0\le b\le
1$. Recall also that $S_ny=\frac12(S_nx+S_n\bar{x})+\pi_{t,y}$.

Therefore
\begin{equation}\label{E:t-1Est}
(S_nz-S_nw)|_{Y_{t-1}}=\left.\left(\left(1-\frac12b-\frac{1-c}2\right)S_n\bar{x}-
\left(c+\frac{1-c}2-\frac12b\right)S_nx\right)\right|_{Y_{t-1}}\end{equation}
Observe that each of the coefficients of $S_n\bar{x}$ and $S_nx$
in this sum is at least
\[\left(\frac12-\left(\frac12\right)^{m_0}\right).\] There are  two possible cases:

(A) $x\in W_{t-1}\backslash W_{t-2}$, $\bar{x}\in W_{t-2}$;

(B) $\bar{x}\in W_{t-1}\backslash W_{t-2}$, $x\in
W_{t-2}$.\medskip

The cases are similar, so we consider the case (A) only. Let $o\in
W_{t-2}$ be such that $x$ is in the subdiamond with diagonal
$o\bar{x}$, so $S_n\bar{x}, S_no\in B_{t-2}$,
\[S_nx=\frac12(S_n\bar{x}+S_n o)+\e\left(\frac12+\e\right)^{t-2}y_{t-1,\s_{t-1}(x)},\]
and
\[(S_nz-S_nw)|_{Y_{t-1}}=-\left(c+\frac{1-c}2-\frac12b\right)\e\left(\frac12+\e\right)^{t-2}y_{t-1,\s_{t-1}(x)}\]
We get

\[\begin{split}\frac{d_{W_n}(w,z)}{||S_nw-S_nz||}&\le
\frac{2\left(\frac12+\ep\right)^t}{\left(\frac12-\ep\right)\frac1{2C}\left(c+\frac{1-c}2-\frac12b\right)\e\left(\frac12+\e\right)^{t-2}}\\
&\le\frac{4C\left(\frac12+\ep\right)^2}{\left(\frac12-\ep\right)\ep\left(\frac12-\left(\frac12\right)^{m_0}\right)}.
\end{split}\]
The obtained number depends only on $C$ and $\ep$. This concludes
the proof.
\end{proof}

\section{More general examples}\label{S:coral}

The goal of this section is to generalize the results of Section
\ref{S:SnowDiam} to more general `hierarchically built weighted
graphs', which we denote  $\{G_i\}_{i=0}^\infty$ and  call  {\it
corals} because they are more chaotic than diamonds.

\begin{definition}\label{D:coral}
We pick  $\la\in\left(\frac12,1\right)$ and a sequence $\{N_i\}_{i=0}^\infty$ of natural numbers so that $N_0=2$ and $N_i\ge 1$ for all $i\ge 1$. The sequence
$\{G_n\}_{n=0}^\infty$ of {\it corals} is defined inductively.
Vertices and edges of a coral come in {\it generations}   denoted $\{V_i\}_{i=0}^\infty$
and $\{E_i\}_{i=0}^\infty$, respectively. We proceed as follows  (see Figure~\ref{coral1} for a sample graph $G_1$):

\begin{itemize}

\item $G_0$ is the same as $D_0$, i.e. $V_0$ consists of two vertices $v_0, v_1$ which are joined by one edge of
weight $1$. Thus $G_0=(V_0,E_0)$, where $|V_0|=2$, $|E_0|=1$.

\item Suppose that $\bigcup_{i=0}^kV_i$, $\bigcup_{i=0}^kE_i$,  and $G_k$ have been already
defined. Let $V_{k+1}$ be a  set of
cardinality $N_{k+1}$, disjoint with $\bigcup_{i=0}^kV_i$. The vertex set of
the graph $G_{k+1}$ is $\bigcup_{i=0}^{k+1}V_i$. The set $E_{k+1}$ of new
edges  is a subset of edges joining the
vertices of $V_{k+1}$ with $\bigcup_{i=0}^kV_i$. Every edge in $E_{k+1}$ is is given
weight $\la^{k+1}$. Edges in $E_{k+1}$ are chosen so that each vertex in
$V_{k+1}$ has degree $1$ or $2$ and if a vertex $v\in V_{k+1}$ has degree $2$ then it is adjacent to vertices $u,w\in
\bigcup_{i=0}^kV_i$ which  are joined by an edge $uw$  in $E_k$, i.e. $uw$ is of
length $\lambda^k$ in $G_k$.
\begin{figure}[h]
\centering
\includegraphics{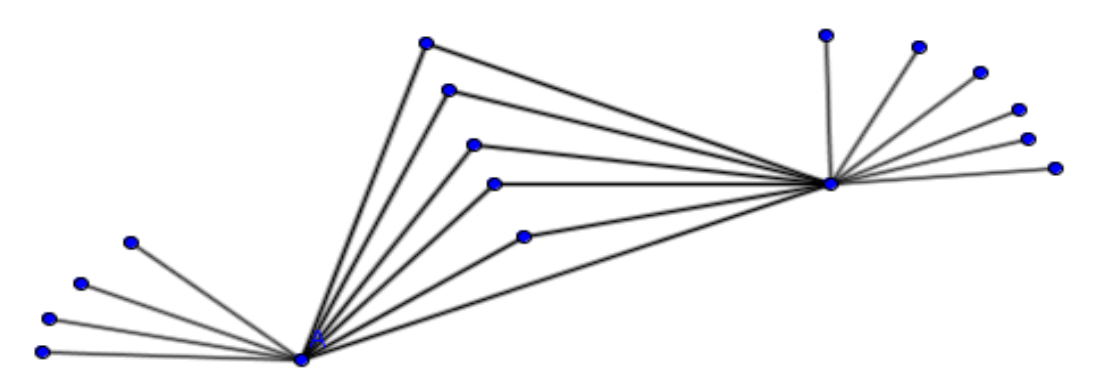}
\caption{Sketch of $G_1$}\label{coral1}
\end{figure}
\end{itemize}
\end{definition}

\begin{remark}
The graph $G_n$ depends on $\lambda$,   the numbers
$N_1,N_2,\dots,N_n$, and  on the  choices that we make when we attach
new vertices to already existing. For brevity, we do not reflect these dependencies in the notation for $G_n$.
\end{remark}

Note that a {\it coral} can be regarded as a chaotically branching
snowflaked diamond in which we allow attaching smaller diamonds to
vertices of larger diamonds. In particular the weighted graph
$W_n$ is an example   of a very regular  coral. Also one can think
of the coral as constructed in a fractal-like fashion, where we
start from  two vertices joined by one edge. On the next step we
replace the one edge by a copy of $G_1$ (see Figure~\ref{coral1})
so the set of edges is now $E_0\cup E_1$, where every edge in
$E_1$ has  length $\la$. We now replace every edge in $E_1$  by a
scaled (by $\la$) version of $G_1$, obtaining set $E_2$ of
additional edges  of length $\la^2$. We continue for arbitrary
number of generations. The   difference between this procedure and
a true fractal is that every scaled copy of $G_1$ can have
different number of vertices and edges, so the final graph can be
very chaotic (see Figure~\ref{coral2}).

\begin{figure}[h]
\centering
\includegraphics{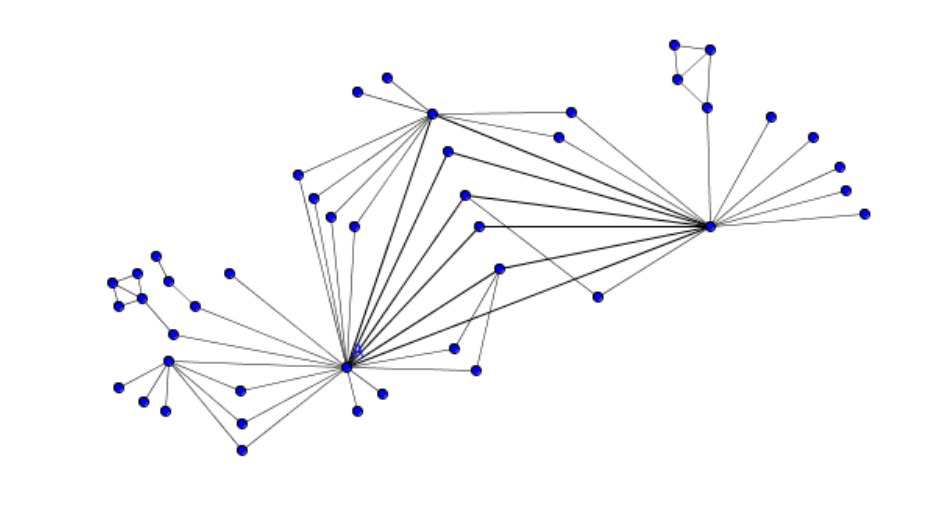}
\caption{An example of a coral with a few
generations}\label{coral2}
\end{figure}

Our goal is to prove that Corollary~\ref{cor:anyemb-old} can be generalized for
corals.

We introduce the following function $L:\mathbb{N}\to \mathbb{N}$,
\[L(i)=\begin{cases} 1 &\hbox{ if } i=1,2\\
2 &\hbox{ if } i=3,4\\
\lceil\log_4 i\rceil &\hbox{ if } i\ge 5.
\end{cases}\]
$L(i)$ shows the dimension  which is sufficient to accommodate $i$
$\de$-separated points, for $\de=1/16$,  in the unit sphere, cf. Lemma~\ref{L:apart}.

\begin{theorem}\label{T:GenTwoWay}
 Let $C\ge 1$ and $\la\in (1/2,1)$. Then there exists a constant $D=D(C,\la)$, so that every   coral $G_n$ with parameters $\la$ and $\{N_i\}_{i=0}^n\subset\bbN$, $D$-embeds into any   Banach space   $X$  which contains a basic sequence of length
$\sum_{i=0}^n L(N_i)$ with basis constant $\le C$.
\end{theorem}

\begin{remark}\label{R:RelWthW_n}
Note that in the case when $G_n=W_n$,  Theorem~\ref{T:GenTwoWay}
reduces to Theorem~\ref{T:EmbAny}.
\end{remark}

\begin{remark}\lb{referee} Analogs
of Theorem~\ref{Wndoesnotembed} and Proposition~\ref{P:EucNoLog}
do not hold for some families of corals. In fact, certain families
of corals can embed in low dimensional Euclidean spaces, for
example a family consisting of a triangle with a progressively
longer tails embeds into $\bbR$ with uniformly bounded
distortions.
\end{remark}

For the proof of Theorem~\ref{T:GenTwoWay} we will need an
analogue of Lemma~\ref{C:le2}.

\begin{lemma}[{This is a version of \cite[Claim 4.1]{Ost14}}]\label{C:le2rep}
 A shortest path between two vertices in $G_n$ can contain edges of each possible
length: $1,\lambda,\lambda^2,\lambda^3, \dots$ at most twice.
Actually for $1$ this can happen only once because there is only
one such edge. If there are two longest edges, they are adjacent.
\end{lemma}

The proof of this lemma is a slightly modified version of the
proof of Lemma~\ref{C:le2}. We start with a definition of a notion
analogous to the notion of a subdiamond.

\begin{definition}\label{D:subcoral}
We define a {\it (degree 2) subcoral}  of a coral $G_n$ grown out
of an edge $uv$ to be the subgraph of $G_n$ induced by the set of
vertices containing $u,v$ and viewed as constructed in steps such
that the following conditions are satisfied:

(1) All vertices except $u$ and $v$ can be included into the
subcoral only if they have degree $2$ when they appear for the
first time;

(2) All vertices which have degree $2$ when they appear, with both
ends in the subcoral, get into the subcoral.

The edge from which a subcoral evolved is called its {\it
diagonal}.
\end{definition}

\begin{proof}[Proof of Lemma~\ref{C:le2rep}]
 Let $e=uv$ be one of the longest edges
in the path and $\lambda^k$ be its length.

For each   edge of the graph $G_n$ except the initial edge, one of
the ends was introduced later. Assuming that $e$ is not the
initial edge of the graph, we may assume that the vertex $v$ was
introduced later than $u$.

There are two cases:

(1) The vertex $v$ was attached to two vertices of an edge $uw$.
Let $S$ be the subcoral which evolved from $uw$, so $S$ contains
$e$ and has a diagonal of length $\lambda^{k-1}$

(2) The vertex $v$ was attached to the vertex $u$ only.
\medskip

The rest of the path consists of two pieces: (1) The one which
starts at $v$; (2) The one which starts at $u$.

We claim that in the case (1) the part which starts at $v$ can
never leave $S$. It obviously cannot leave through $u$, it cannot
leave through the $w$ also, because otherwise the piece of the
path between $u$ and $w$ could be replaced by the diagonal of $S$,
which is strictly shorter.\medskip

This implies that the part of the path in $S$ which starts at $v$
can contain edges only strictly shorter than $\lambda^k$.
\medskip

The same is true in the case (2) because only vertices of further
generations will be attached to $v$ in this case and they are
attached using edges of length $\le\lambda^{k+1}$.
\medskip

For the next edge in the part of the path which starts at $v$ we
can repeat the argument and get (by induction) that lengths of
edges in the remainder of the path are strictly
decreasing.\medskip

The part of the path which starts at $u$ can be considered
similarly.
\medskip

The last statement of the Lemma is immediate from the proof.
\end{proof}

\begin{proof}[Proof of Theorem \ref{T:GenTwoWay}]
\remove{Let $C\ge 1$ and $\la\in (1/2,1)$. Then there exists a
constant $D=D(C,\la)$, so that every   coral $G_n$ with parameters
$\la$ and $\{N_i\}_{i=0}^n\subset\bbN$, $D$-embeds into any Banach
space   $X$  which contains a basic sequence of length
$\sum_{i=0}^n L(N_i)$ with basis constant $\le C$.} The proof is
very similar to the proof of Theorem~\ref{T:EmbAny}.

Let $L=\sum_{i=0}^n L(N_i)$,   $X$ be a Banach space and  $\{x_i\}_{i=0}^{L-1}$ be a basic sequence in
$X$ with $\|x_i\|=1$ for all $i$, and a basis constant $\le C$.
Let $Y_0=\Span\{x_0\}$, and
$Y_m=\Span\{x_j :
\left(\sum_{k=0}^{m-1}L(k)\right) +1 \le j \le \sum_{k=0}^{m}L(k)\}$,
for $m=1,\dots,n$. Thus $\dim Y_m=L(N_m)$ for $m=0,\dots,n$. Let
$\{y_{m,k}\}_{k=1}^{N_m}$ be elements of the unit sphere of $Y_m$
satisfying the conditions of Lemma~\ref{L:apart} with $\de=1/16$
(observe that the definition of $L(N_m)$ is such that it is always
possible). Note that for $m> m'$, $y_{m,k}$ and $y_{m',k'}$ are
supported on disjoint intervals with respect to the basis
 $\{x_i\}$, hence
\begin{equation}\label{E:big2}
\|y_{m,k}-y_{m',k'}\|\ge\frac1{C}\|y_{m',k'}\|=\frac1{C}.
\end{equation}

We construct an embedding $T:G_n\to X$ in the following way. We
define it in steps for vertices of $V_0,V_1,\dots,V_n$

\begin{itemize}

\item The map $T$ maps the two vertices of $V_0$  to $0$ and
$x_0$, respectively. It is clear that it is an isometric
embedding.

\item Suppose that we have already constructed the restriction of
$T$ to $\bigcup_{i=0}^{m-1}V_i$. Our next step is to extend $T$ to
$V_m$. Observe that our notation is such that there exists a
bijection between $V_m$ and $\{y_{m,k}\}_{k=1}^{N_m}$. Let $w\in V_m$. We denote
the vector corresponding to a vertex $w$ by $y_{m,\sigma_m(w)}$.
Then we define $Tw$ as follows

\begin{itemize}

\item if the vertex $w$ is attached to two vertices $u, v\in \bigcup_{i=0}^{m-1}V_i$, we
let
\begin{equation}\label{E:DefTdeg2}
Tw=\frac12(Tu+Tv)+\left(\lambda-\frac12\right)\lambda^{m-1}y_{m,\s_m(w)}.
\end{equation}

\item  if the vertex $w$ is attached to one vertex $u\in \bigcup_{i=0}^{m-1}V_i$, we let
\begin{equation}\label{E:DefTdeg1}
Tw=Tu+\lambda^my_{m,\s_m(w)}.
\end{equation}

\end{itemize}
\end{itemize}

Now we estimate the distortion of $T$. First we show that the map
$T$ is $1$-Lipschitz. This can be proved for $T|_{G_m}$ by
induction on $m=0,1,\dots,n$ (observe that the metric induced on
$G_m$ from $G_{m+1}$ coincides with the metric of $G_m$). It
suffices to prove that for each $w\in V_{m}$ and each edge $uw$
in $G_{m}$ we have $d_W(u,w)=\lambda^{m}$ and
$||Tu-Tw||\le\lambda^{m}$.

The equality $d_W(u,w)=\lambda^{m}$
follows immediately from our definitions.
To prove that
$||Tu-Tw||\le\lambda^{m}$, we need to consider two cases: (a) $w$
has degree $1$ in $G_{m}$; (b) $w$ has degree $2$ in $G_{m}$.

Since $\|y_{m,\s_m(w)}\|=1$,  the desired inequality in the case
(a) follows immediately  from \eqref{E:DefTdeg1}. In the case (b)
we get
\[||Tu-Tw||\le
\frac12||Tu-Tv||+\left(\lambda-\frac12\right)\lambda^{m-1}\le\lambda^{m},\]
where for the last inequality we use the assumption that $uv$ is
an edge in $G_{m-1}$ and therefore $||Tu-Tv||\le\lambda^{m-1}$.

To estimate  the Lipschitz constant of $T^{-1}$ from above we
consider any shortest path $P$ between two vertices $w,z$ in
$G_n$. Let $\lambda^t$ be the length of the longest edge in it. By
Lemma~\ref{C:le2rep},
\begin{equation}\label{E:dAbove2}
d_W(w,z)\le
2\lambda^t\left/\left(1-\lambda\right)\right..
\end{equation}

On the other hand, since the subspaces $Y_m$ are supported on
disjoint intervals with respect to the basis $\{x_i\}$, we have
for every $m\in\{0,1,\dots,n\}$,
\begin{equation}\label{E:FunctBelow2}
||Tw-Tz||\ge \frac1{2C} \|(Tw-Tz)\big|_{Y_m}\|,
\end{equation}
 where by $x|_{Y_m}$ we denote the
natural projection of $x$ onto $Y_m$.

 Let $m_0\in
\mathbb{N}$ be the smallest number is such that
\begin{equation}\label{m02} \lambda+\lambda^2+\dots+\lambda^{m_0}> 1+\lambda^{m_0}.\end{equation} Such number
$m_0$ obviously exists since $\lambda>\frac12$. It is clear that
$m_0\ge 3$ since $\lambda<1$, and that $m_0$ depends only on
$\lambda$.

\remove{ Note that if $xy$ is an  edge of length $\lambda^t$, then
  \begin{equation}\label{oneedge}
\|S_nx-S_ny\|\ge \frac1{2C}\e\lambda^{t-1}.
\end{equation}

Indeed, this is obvious if $t=0$. If $t\ge 1$, one of the
vertices, say $y$, is of a later generation than the other vertex.
Then $y\in W_t\setminus W_{t-1}$, $x\in W_{t-1}$, and there exists
a vertex $\bar x\in W_{t-1}$ so that $xy$ is an edge in a
subdiamond with diagonal $\bar xy$, which is of length
$\lambda^{t-1}$ and
$$S_ny=\frac12(S_nx+S_n\bar x)+\e\lambda^{t-1}y_{t,\s_t(y)},$$
where $S_nx, S_n\bar x\in\Span\{\bigcup_{m<t} Y_m\}$ and
$y_{t,\s_t(y)}\in Y_t$. Thus \[\begin{split}\|S_nx-S_ny\|&\ge
\frac1{2C} \left\|(S_nx-S_ny)|_{Y_m}\right\|\\&= \frac1{2C}
\left\|\e\lambda^{t-1}y_{t,\s_t(y)}\right\|\\&=\frac1{2C}\e\lambda^{t-1}.\end{split}\]}

Now we turn to estimates of $||Tw-Tz||$ from below. Let $xy$ be
one of the edges of the largest length $\lambda^t$ in the path $P$
from $w$ to $z$ (by Lemma~\ref{C:le2rep} we know that the path $P$
contains at most two such edges; and that if there are two of
them, they share a vertex). \Buo\ we assume that $y\in V_t$ and
$x\in\cup_{i=0}^{t-1} V_i$.

{\bf (1)} In the case where $y$ is of degree $2$ in $G_t$ let
$\bar x$ be the vertex in $V_{t-1}$ so that $x\bar x$ is an edge
of the length $\lambda^{t-1}$ in $G_{t-1}$ and $\bar xy$ is an
edge of length $\lambda^t$ in $G_t$. Then the part of the path $P$
from $y$ to $z$ does not contain an edge of length $\lambda^t$.
Furthermore, some part of this path (from $y$ to $z$), starting at
$y$ (possibly all of the path from $y$ to $z$) is either in the
subcoral with diagonal $yx$, or in the subcoral with diagonal
$y\bar{x}$, and then leaves for parts of the coral which are
attached to older parts of the coral through one vertex. Let $\bar
z$ be the vertex at which this happens. We let $\bar z=z$ if this
never happens.

{\bf (2)} In the case where $y$ is of degree $1$ in $G_t$ we do
the same, but in this case the only option which is available is
the option of subcoral with the diagonal $xy$.
\medskip

Similarly we define $\bar w$. For simplicity we denote the vector
$\left(\lambda-\frac12\right)\lambda^{t-1}y_{t,\sigma_t(y)}$ by
$\pi_{t,y}$.

\begin{lemma}\label{L:Try2} {\bf (i)} If $\bar z$ is in the subcoral with the
diagonal $y\bar{x}$, then
\[Tz|_{Y_t}=T\bar z|_{Y_t}=\alpha\pi_{t,y}\] for some
$\alpha\ge\left(\frac12\right)^{m_0-1}$.

{\bf (ii)} If $\bar z$ is in the subcoral with diagonal $yx$, then
\[(Ty-Tz)|_{Y_t}=(Ty-T\bar z)|_{Y_t}=\beta\pi_{t,y}\] for some
$0\le\beta\le\left(\frac12\right)^{m_0-1}$.

{\bf (iii)} If $\bar w$ is in the subcoral with diagonal $yx$,
then
\[Tw|_{Y_t}=T\bar w|_{Y_t}=\gamma\pi_{t,y}\] for some
$0\le\gamma\le\left(\frac12\right)^{m_0-1}$.

{\bf (iv)} If $\bar w$ is not in the subcoral with the diagonal
$yx$, then
$$Tw|_{Y_t}=T\bar w|_{Y_t}=\omega y_{t,k},$$
 for some
$k\ne\sigma_t(y)$ and $\omega\in[0,1]$.
\end{lemma}

\begin{proof} Observe that the leftmost equalities in each of the
statements follow immediately from \eqref{E:DefTdeg1}, so we shall
focus only on the rightmost equalities.
\medskip

{\bf (i)} Let $\bar z$ be in the subcoral with diagonal $y\bar x$.
Observe that ends of edges of length $\le\lambda^{m_0+t-1}$ with
one end at $\bar{x}$ and the other end in the subcoral with the
diagonal $\bar{x}y$ cannot be in $P$ because then, by \eqref{m02},
the path through $\bar{x}$ would be shorter. Therefore
\begin{equation*}%\label{Snwp}
T\bar z=(1-b)T\bar{x} +bTy + \bar z_t,
\end{equation*}
where  $Ty \in B_{t}\DEF\Span\{x_j: j\le\sum_{k=0}^t L(N_k)
\}=\Span(\bigcup_{k=0}^t Y_k)$, $T\bar x\in B_{t-1}$,  $\bar
z_t\in T_t\DEF\Span\{x_j: j>\sum_{k=0}^t L(N_k)\}$, and $b\ge
\left(\frac12\right)^{m_0-1}$. Note that
\begin{equation}\label{Sny2}
Ty=\frac12(Tx+T\bar x)+\pi_{t,y},
\end{equation}
where  $Tx, T\bar x \in B_{t-1}$ and $\pi_{t,y}\in Y_t$. Hence
\begin{equation*}%\label{Snw}
T\bar z|_{Y_t}=b\pi_{t,y},
\end{equation*}
the conclusion follows.
\medskip

\noindent{\bf (ii)} If  $\bar z$ is in the subcoral with diagonal
$yx$, since $yx$ is a part of a shortest path, we conclude that
 the longest edge in the part of the path $P$ from $y$ to $\bar z$ has length
$\le\lambda^{t+m_0}$, where $m_0$ satisfies \eqref{m02}. Thus
\begin{equation}\label{Snz12}
T\bar z= aTy +(1-a)Tx + z_t,
\end{equation}
where $Ty \in B_t$, $Tx \in B_{t-1}$,  $z_t\in T_t$, and
\begin{equation*}%\label{heighta}
1\ge a \ge 1- \sum_{k=m_0}^\infty\left(\frac12\right)^{k}=
1-\left(\frac12\right)^{m_0-1}.
\end{equation*}

By \eqref{Sny2} we get
\[(Ty-T\bar z)|_{Y_t}=\pi_{t,y}-a\pi_{t,y}=(1-a)\pi_{t,y},\]
the conclusion follows.\medskip

\noindent{\bf (iii)} If $\bar w$ is in the subcoral for which $xy$
is the diagonal, similarly as in \eqref{Snz12}, we obtain
$$T\bar w= cTx +(1-c)Ty + w_t,$$
where $Tx\in B_{t-1}, Ty \in B_t$, $w_t\in T_t$, and $ 1\ge c\ge
1-\left(\frac12\right)^{m_0-1}$. By \eqref{Sny2},
\[T\bar w|_{Y_t}=(1-c)\pi_{t,y},\]
and we are done in this case.

\noindent{\bf (iv)} If $\bar w$ is not in the subcoral for which
$xy$ is the diagonal, let $q\in V_t$, $q\ne y$, be the vertex
which is an endpoint of an edge of length $\lambda^{t}$ which is a
diagonal of the subcoral that contains $\bar w$. By construction,
the projection of $T\bar w$ onto the subspace $Y_t$ is a multiple
of $y_{t,\sigma_t(q)}\ne y_{t,\sigma_t(y)}$, with some coefficient
$\omega\in[0,1]$.\end{proof}

\remove{ Thus, by \eqref{Snz2} and Lemma~\ref{L:apart},
\begin{equation*}%\label{big2}
\begin{split}
\|S_nz-S_nw\|&\ge \frac1{2C} \|(S_nw-S_nz)|_{Y_t}\|\\
 &= \frac1{2C} \left\|\g\e\left(\frac12+\e\right)^{t-1} y_{t,\s_t(y)}-\al y_{t,\s_t(q)}\right\|\\
&\ge
\frac1{2C}\frac{\de}2\max\left(\g\e\left(\frac12+\e\right)^{t-1},\al\right)
\\&\ge  \frac1{4C}\e\de\left(\frac12+\e\right)^{t-1}\g_1(\e),
\end{split}
\end{equation*}
which finishes the proof.}

Observe that Lemma \ref{L:Try2} implies the estimate for the
Lipschitz constant of $T^{-1}$, and thus Theorem
\ref{T:GenTwoWay}, in all of the cases except the case where both
{\bf (i)} and {\bf (iii)} hold. Consider, for example the case
where {\bf (i)} and {\bf (iv)} hold. Then (we use
\eqref{E:dAbove2}, \eqref{E:FunctBelow2}, the conclusions of {\bf
(i)} and {\bf (iv)}, the definition of $\pi_{t,y}$ and the
moreover statement in Lemma \ref{L:apart})

\[\begin{split}\frac{d_W(w,z)}{||S_nw-S_nz||}&\le\frac{2\lambda^t}{\left(1-\lambda\right)\frac1{2C}||\alpha\pi_{t,y}-\omega
y_{t,k}||}\\
&\le\frac{4C\lambda^t}{\left(1-\lambda\right)\frac{\delta}2\alpha\left(\lambda-\frac12\right)\lambda^{t-1}}=
\frac{8C\lambda}{\left(1-\lambda\right)\delta\left(\lambda-\frac12\right)\left(\frac12\right)^{m_0-1}},
\end{split}\]
and this number depends only on $C$ and $\lambda$.\medskip

 It remains to consider the  case where both {\bf (i)} and {\bf
(iii)} hold. In this case we estimate from below the norm of
$(T\bar w-T\bar z)|_{Y_{t-1}}$. We use $T\bar z=(1-b)T\bar{x} +bTy
+ \bar z_t$ and $T\bar w= cTx +(1-c)Ty + w_t$ with $1\ge b\ge
\left(\frac12\right)^{m_0-1}$ and $1\ge c \ge
1-\left(\frac12\right)^{m_0-1}$. The value of $b$ actually does
not matter for our argument, it is only important that $0\le b\le
1$. Recall also that $Ty=\frac12(Tx+T\bar{x})+\pi_{t,y}$.

Therefore
\begin{equation}\label{E:t-1Est2}
(T\bar z-T\bar
w)|_{Y_{t-1}}=\left.\left(\left(1-\frac12b-\frac{1-c}2\right)T\bar{x}-
\left(c+\frac{1-c}2-\frac12b\right)Tx\right)\right|_{Y_{t-1}}\end{equation}
Observe that each of the coefficients of $T\bar{x}$ and $Tx$ in
this sum is at least
\[\left(\frac12-\left(\frac12\right)^{m_0}\right).\] There are  two possible cases:

(A)  $x\in V_{t-1}$, $\bar x\in\cup_{i=0}^{t-2}V_i$;\medskip

(B) $\bar{x}\in V_{t-1}$, $x\in\cup_{i=0}^{t-2}V_i$.\medskip

The cases are similar, so we consider the case (A) only.
\medskip

{\bf Subcase (1):} There is $o\in\cup_{i=0}^{t-2}V_i$ such that
$x$ is in the subcoral with diagonal $o\bar{x}$, so $T\bar{x},
To\in B_{t-2}$, and
\[Tx=\frac12(T\bar{x}+T o)+\left(\lambda-\frac12\right)\lambda^{t-2}y_{t-1,\s_{t-1}(x)},\]
and
\[(T\bar z-T\bar w)|_{Y_{t-1}}=-\left(c+\frac{1-c}2-\frac12b\right)\left(\lambda-\frac12\right)\lambda^{t-2}y_{t-1,\s_{t-1}(x)}.\]
We get

\[\begin{split}\frac{d_W(w,z)}{||T\bar z-T\bar w||}&\le
\frac{2\lambda^t}{\left(1-\lambda\right)\frac1{2C}\left(c+\frac{1-c}2-\frac12b\right)\left(\lambda-\frac12\right)\lambda^{t-2}}\\
&\le\frac{4C\lambda^2}{\left(1-\lambda\right)\left(\lambda-\frac12\right)\left(\frac12-\left(\frac12\right)^{m_0}\right)}.
\end{split}\]
The obtained number depends only on $C$ and $\lambda$.

{\bf Subcase (2):} The vertex $x$ has degree $1$ in $G_{t-1}$ so
$T\bar{x}\in B_{t-2}$,
\[Tx=T\bar{x}+\lambda^{t-1}y_{t-1,\s_{t-1}(x)},\]
and
\[(T\bar z-T\bar w)|_{Y_{t-1}}=-\left(c+\frac{1-c}2-\frac12b\right)\lambda^{t-1}y_{t-1,\s_{t-1}(x)}.\]
We get

\[\begin{split}\frac{d_W(w,z)}{||T\bar z-T\bar w||}&\le
\frac{2\lambda^t}{\left(1-\lambda\right)\frac1{2C}\left(c+\frac{1-c}2-\frac12b\right)\lambda^{t-1}}\\
&\le\frac{4C\lambda}{\left(1-\lambda\right)\left(\frac12-\left(\frac12\right)^{m_0}\right)}.
\end{split}\]
The obtained number depends only on $C$ and $\lambda$. This
concludes the proof.
\end{proof}

\section*{Acknowledgement}

The authors would like to thank Gideon Schechtman for suggesting
the problem and sharing relevant information. Also we would like
to thank William B.~Johnson and Stanis\l aw Szarek for useful
information and the referee for suggesting improvements to our
paper. The first-named author gratefully acknowledges the support
by NSF DMS-1201269.

\end{large}

\renewcommand{\refname}{

\textsc{Department of
Mathematics and Computer Science, St. John's University, 8000 Utopia Parkway, Queens, NY 11439, USA} \par
  \textit{E-mail address}: \texttt{ostrovsm@stjohns.edu} \par

\textsc{Department of Mathematics, Miami University,
Oxford, OH 45056, USA} \par
  \textit{E-mail address}: \texttt{randrib@miamioh.edu} \par


\section{References}}

\begin{thebibliography}{WWW12}

\bibitem[AB14]{AB14} F.~Albiac, F.~Baudier, Embeddability of snowflaked metrics with
applications to the nonlinear geometry of the spaces $L_p$ and
$\ell_p$ for $0<p<\infty$, {\it J. Geom. Anal.} {\bf 25} (2015),
no. 1, 1--24; DOI 10.1007/s12220-013-9390-0

\bibitem[AM83]{AM83} N.~Alon, V.\,D.~Milman, Embedding of $\ell_\infty^k$ in
finite-dimensional Banach spaces. {\it Israel J. Math.} {\bf 45}
(1983), no. 4, 265--280.

\bibitem[ABV98]{ABV98} J.~Arias-de-Reyna, K.~Ball, R.~Villa, Concentration of the distance in finite-dimensional normed
spaces.  {\it Mathematika} {\bf  45}  (1998),  no. 2, 245--252.

\bibitem[Ass83]{Ass83} P.~Assouad, Plongements lipschitziens dans $\mathbb{R}^n$.
{\it Bull. Soc. Math. France} {\bf 111} (1983), no. 4, 429--448.

\bibitem[Bar96]{Bar96}
Y.~Bartal, Probabilistic approximation of metric spaces and its
algorithmic applications,
  in {\em The\/}  37th {\em Annual Symposium on Foundations of
Computer Science}, pp.\  184--193, 1996,
  IEEE Comput.\ Sci.\ Press, Los Alamitos, CA, 1996.

\bibitem[Bar99]{Bar99} Y.~Bartal, On approximating arbitrary metrices by tree
metrics. {\it STOC '98} (Dallas, TX), 161--168, ACM, New York,
1999.

\bibitem[BLMN04]{BLMN04} Y.~Bartal, N.~Linial, M.~Mendel, A.~Naor, Low
dimensional embeddings of ultrametrics. {\it European J. Combin.}
{\bf 25} (2004), no. 1, 87--92.

\bibitem[BLMN05]{BLMN05} Y.~Bartal,
N.~Linial, M.~Mendel, A.~Naor, On metric Ramsey-type phenomena,
{\it Annals of Math.}, {\bf 162} (2005), 643--709.

\bibitem[BM04]{BM04} Y.~Bartal, M.~Mendel, Dimension reduction for ultrametrics.
{\it Proceedings of the Fifteenth Annual ACM-SIAM Symposium on
Discrete Algorithms}, 664--665, ACM, New York, 2004.

\bibitem[BB91]{BB91} J.~Bastero, J.~Bernu\'es,
Applications of deviation inequalities on finite metric sets. {\it
Math. Nachr.} {\bf 153} (1991), 33--41.

\bibitem[BBK89]{BBK89} J.~Bastero, J.~Bernu\'es, N.~Kalton, Embedding
$\ell^n_\infty$-cubes in finite-dimensional 1-subsymmetric spaces.
Congress on Functional Analysis (Madrid, 1988). {\it Rev. Mat.
Univ. Complut. Madrid} {\bf 2} (1989), suppl., 47--52.

\bibitem[BPS95]{BPS95} J.~Bastero, A.~Pe{\~{n}}a, G.~Schechtman, Embedding $\ell^n_\infty$-cubes in
low-dimensional Schatten classes.
 {\it Geometric aspects of functional analysis} (Israel, 1992--1994), 5--11, {\it Oper. Theory Adv. Appl.}, {\bf 77}, Birkh�user, Basel, 1995.

\bibitem[BFM86]{BFM86} J.~Bourgain, T.~Figiel, V.~Milman,  On Hilbertian subsets of
finite metric spaces. {\it Israel J. Math.} {\bf 55} (1986), no.
2, 147--152.

\bibitem[DS13]{DS13} G.~David, M.~Snipes,  A non-probabilistic proof of the
Assouad embedding theorem with bounds on the dimension. {\it Anal.
Geom. Metr. Spaces} {\bf 1} (2013), 36--41.

\bibitem[Dvo61]{Dvo61} A.~Dvoretzky, Some results on convex bodies and Banach spaces, in:
{\it Proc. Internat. Sympos. Linear Spaces} (Jerusalem, 1960),
pp.~123--160, Jerusalem Academic Press, Jerusalem; Pergamon,
Oxford, 1961; Announcement: A theorem on convex bodies and
applications to Banach spaces, {\it Proc. Nat. Acad. Sci. U.S.A.},
{\bf 45} (1959) 223--226; erratum, 1554.

\bibitem[FL94]{FL94}
Z. F{\"u}redi, P. A. Loeb,  On the best constant for the
Besicovitch covering theorem,  {\it Proc. Amer. Math. Soc.} {\bf
121} (1994), no. 4, 1063--1073.

\bibitem[Glu81]{Glu81} E.\,D.~Gluskin, The diameter of the Minkowski compactum is
roughly equal to $n$. (Russian) {\it Funktsional. Anal. i
Prilozhen.} {\bf 15} (1981), no. 1, 72--73.

\bibitem[GNRS04]{GNRS04} A.~Gupta, I.~Newman, Y.~Rabinovich,
A.~Sinclair, Cuts, trees and $\ell_1$-embeddings of graphs, {\it
Combinatorica}, {\bf 24} (2004) 233--269; Conference version in:
{\it 40th Annual IEEE Symposium on Foundations of Computer
Science}, 1999, pp.~399--408.

\bibitem[Hei01]{Hei01}
J.\,M.~Heinonen, {\it Lectures on Analysis on Metric Spaces},
Universitext, Springer-Verlag, New York, 2001.

\bibitem[Hei03]{Hei03}
J.\,M.~Heinonen, {\it Geometric embeddings of metric spaces}.
Report. University of Jyv\"{a}s\-kyl\"{a} Department of
Mathematics and Statistics, 90. University of Jyv\"{a}skyl\"{a},
Jyv\"{a}skyl\"{a}, 2003; Available at {\tt
http://www.math.jyu.fi/research/reports/rep90.pdf}.

\bibitem[Hug12]{Hug12} B.~Hughes,
Trees, ultrametrics, and noncommutative geometry. {\it Pure Appl.
Math. Q.} {\bf 8} (2012), no. 1, 221--312.

\bibitem[JS09]{JS09} W.\,B.~Johnson, G. Schechtman,
Diamond graphs and super-reflexivity, {\it J. Topol. Anal.}, {\bf
1} (2009), no. 2, 177--189.

\bibitem[LP01]{LP01} U.~Lang, C.~Plaut, Bilipschitz embeddings
of metric spaces into space forms. {\it Geom. Dedicata} {\bf 87}
(2001), no. 1--3, 285--307.

\bibitem[MT93]{MT93} P.~Mankiewicz, N.~Tomczak-Jaegermann,
Embedding subspaces of $\ell_\infty^n$ into spaces with Schauder
basis. {\it Proc. Amer. Math. Soc.} {\bf 117} (1993), no. 2,
459--465.

\bibitem[MN13]{MN13} M.~Mendel, A.~Naor, Ultrametric subsets with large
Hausdorff dimension. {\it Invent. Math.} {\bf 192} (2013), no. 1,
1--54.

\bibitem[Mil71]{Mil71} V.\,D.~Milman,
A new proof of A. Dvoretzky's theorem on cross-sections of convex
bodies. (Russian) {\it Funktsional. Anal. i Prilozhen.} {\bf 5}
(1971), no. 4, 28--37.

\bibitem[Mil85]{Mil85} V.\,D.~Milman, Almost Euclidean quotient spaces of subspaces of
a finite-dimensional normed space. {\it Proc. Amer. Math. Soc.}
{\bf 94} (1985), no. 3, 445--449.

\bibitem[NN12]{NN12} A.~Naor, O.~Neiman, Assouad's theorem
with dimension independent of the snowflaking. {\it Rev. Mat.
Iberoam.} {\bf 28} (2012), no. 4, 1123--1142.

\bibitem[Ost14]{Ost14} M.\,I.~Ostrovskii,
Metric characterizations of superreflexivity in terms of word
hyperbolic groups and finite graphs, {\it Anal. Geom. Metr.
Spaces} {\bf 2} (2014), 154--168.

\bibitem[Rud95]{Rud95} M.~Rudelson, Estimates of the weak distance
between finite-dimensional Banach spaces. {\it Israel J. Math.}
{\bf 89} (1995), no. 1-3, 189--204.

\bibitem[Sch84]{Sch84} W.\,H.~Schikhof, {\it Ultrametric calculus. An introduction to $p$-adic
analysis.} Cambridge Studies in Advanced Mathematics, {\bf 4}.
Cambridge University Press, Cambridge, 1984.

\bibitem[Sem99]{Sem99} S.~Semmes, Bilipschitz embeddings of metric spaces into
Euclidean spaces. {\it Publ. Mat.} {\bf 43} (1999), no. 2,
571--653.

\bibitem[Shk04]{Shk04} S.\,A.~Shkarin, Isometric embedding of finite
ultrametric spaces in Banach spaces. {\it Topology Appl.} {\bf
142} (2004), no. 1--3, 13--17.

\bibitem[Sza83]{Sza83} S.\,J.~Szarek, The finite-dimensional basis problem with
an appendix on nets of Grassmann manifolds. {\it Acta Math.} {\bf
151} (1983), no. 3--4, 153--179.

\bibitem[ST09]{ST09} S.\,J.~Szarek, N.~Tomczak-Jaegermann,
On the nontrivial projection problem. {\it Adv. Math.} {\bf 221}
(2009), no. 2, 331--342.

\bibitem[Tal95]{Tal95} M.~Talagrand, Embedding of $\ell_\infty^k$ and a theorem of Alon and
Milman. {\it Geometric aspects of functional analysis} (Israel,
1992--1994), 289--293, {\it Oper. Theory Adv. Appl.}, {\bf 77},
Birkh\"auser, Basel, 1995.
\end{thebibliography}
\end{document}